\documentclass{siamltex1213}
\usepackage{amssymb}
\usepackage{amsmath}
\usepackage{overpic} 
\usepackage{tikz} 
\usetikzlibrary{shapes.misc}

\usepackage{enumitem}

{\rm \trivlist \item[\hskip \labelsep{\bf Proof}]}%
{\hspace*{\fill}$\Box$\endtrivlist}
{\rm \trivlist \item[\hskip \labelsep{\bf Proof.} \,]}%
{\hspace*{\fill}$\Box$\endtrivlist}

\def \P{{\sf I\kern-1.5ptP}}

\def \wt{\rightarrow^{\kern-8.5pt*}~}
\def \displacementRank{\nu}

\title{On the singular values of matrices with displacement structure}
\author{Bernhard Beckermann\thanks{Laboratoire Paul Painlev\'e UMR 8524 CNRS, Equipe ANO-EDP, UFR Math\'ematiques, UST Lille, F-59655 Villeneuve d'Ascq CEDEX, France. (\texttt{Bernhard.Beckermann@univ-lille1.fr}) Supported in part by the Labex CEMPI (ANR-11-LABX-0007-01).} \and Alex Townsend\thanks{Department of Mathematics, Cornell University, Ithaca, NY 14853. (\texttt{townsend@cornell.edu}) This work is supported by National Science Foundation grant No.~1522577.} }
\date{\today}
\begin{document}
\maketitle
\begin{abstract}
Matrices with displacement structure such as Pick, Vandermonde, and Hankel matrices appear in a diverse range of applications. 
In this paper, we use an extremal problem involving rational functions to derive explicit bounds on the singular values of such 
matrices. For example, we show that the $k$th singular value of a real $n\times n$ positive definite Hankel matrix, $H_n$, is 
bounded by $C\rho^{-k/\log n}\|H\|_2$ with explicitly given constants $C>0$ and $\rho>1$, where $\|H_n\|_2$ is 
the spectral norm. This means that a real $n\times n$ positive definite Hankel matrix can be approximated, up to an 
accuracy of $\epsilon\|H_n\|_2$ with $0<\epsilon<1$, by a rank $\mathcal{O}(\log n\log(1/\epsilon) )$ matrix. Analogous results are obtained for Pick, Cauchy, real 
Vandermonde, L\"{o}wner, and certain Krylov matrices.
\end{abstract} 



\section{Introduction}
Matrices with rapidly decaying singular values frequently appear in computational mathematics. Such matrices are 
numerically of low rank and this is exploited in applications such as particle simulations~\cite{Greengard_87_01}, 
model reduction~\cite{Antoulas_02_01}, boundary element methods~\cite{Hackbusch_99_01}, and matrix completion~\cite{Candes_09_01}. 
However, it can be theoretically challenging to fully explain why low rank techniques 
are so effective in practice. In this paper, we derive explicit bounds 
on the singular values of matrices with displacement structure and in doing so justify many of the low rank 
techniques that are being employed on such matrices. 

Let $X\in\mathbb{C}^{m\times n}$ with $m\geq n$, $A\in\mathbb{C}^{m\times m}$, and $B\in\mathbb{C}^{n\times n}$, we say that $X$ has 
an $(A,B)$-displacement rank of $\displacementRank$ if $X$ satisfies the Sylvester matrix equation given by
\begin{equation}
A X - X B = MN^*,
\label{eq:SylvesterEquation} 
\end{equation} 
for some matrices $M\in\mathbb{C}^{m\times \displacementRank}$ and $N\in\mathbb{C}^{n\times \displacementRank}$. Matrices with displacement structure include 
Toeplitz ($\displacementRank=2$), Hankel ($\displacementRank=2$), Cauchy ($\displacementRank=1$), Krylov ($\displacementRank=1$), and Vandermonde ($\displacementRank=1$) matrices, as well as Pick ($\displacementRank=2$), 
Sylvester ($\displacementRank=2$), and L\"{o}wner ($\displacementRank=2$) matrices. Fast algorithms for computing matrix-vector products and for 
solving systems of linear equations can be derived for many of these matrices by exploiting~\eqref{eq:SylvesterEquation}~\cite{Heinig_84_01,Morf_1974_01}. 

In this paper, we use the displacement structure to derive explicit bounds on the singular values 
of matrices that satisfy~\eqref{eq:SylvesterEquation} by using an extremal problem for rational functions from complex approximation theory. 
In particular, we prove that the following inequality holds (see Theorem~\ref{thm:ratio}): 
\begin{equation}
\sigma_{j+\displacementRank k}(X) \leq Z_k(E,F)\sigma_j(X), \qquad 1\leq j+\displacementRank k\leq n,
\label{eq:SingularValuesBoundedByZolotarev} 
\end{equation} 
where $\sigma_1(X),\ldots,\sigma_n(X)$ denote the singular values of $X$ and $Z_k(E,F)$ is the
Zolotarev number~\eqref{eq:zolotarev} for complex sets $E$ and $F$ that depend on $A$ and $B$.   Researchers have previously exploited the 
connection between the Sylvester matrix equation and Zolotarev numbers for selecting algorithmic parameters in the Alternating Direction Implicit (ADI) 
method~\cite{b11,Birkhoff_62_01,Lebedev_77_01}, and others have demonstrated that the singular values of matrices satisfying certain
Sylvester matrix equations have rapidly decaying singular values~\cite{Antoulas_02_01,Baker_15_01,penzl}. Here, we derive explicit bounds on all 
the singular values of structured matrices.  Table~\ref{tab:SummaryTable} summarizes our main singular value bounds.
\begin{table} 
\centering
\begin{tabular}{cccc}
Matrix class & Notation & Singular value bound & Ref.\\ 
\hline 
\rule{0pt}{1em}Pick & $P_n$ & $\sigma_{1+2k}(P_n) \leq C_1\rho_1^{-k}\|P_n\|_2$ & Sec.~\ref{sec:PickMatrices} \\[5pt]
Cauchy & $C_{m,n}$ & $\sigma_{1+k}(C_{m,n}) \leq C_2\rho_2^{-k}\|C_{m,n}\|_2$ & Sec.~\ref{sec:CauchyMatrices} \\[5pt]
L\"{o}wner &$L_n$ & $\sigma_{1+2k}(L_n) \leq C_3\rho_3^{-k}\|L_n\|_2$ & Sec.~\ref{sec:LoewnerMatrices} \\[5pt]
Krylov, Herm.~arg.~& $K_{m,n}$ & $\!\!\sigma_{1+2k}(K_{m,n}) \leq C_4\rho_4^{-k/\log n}\|K_{m,n}\|_2\!\!$ & Sec.~\ref{sec:VandermondeMatrices}\\[5pt]
Real Vandermonde &$V_{m,n}$ & $\!\sigma_{1+2k}(V_{m,n}) \leq C_5\rho_5^{-k/\log n}\|V_{m,n}\|_2\!$ & Sec.~\ref{sec:VandermondeMatrices} \\[5pt]
Pos.~semidef.~Hankel &$H_n$ & $\sigma_{1+2k}(H_n) \leq C_6\rho_6^{-k/\log n}\|H_n\|_2$ & Sec.~\ref{sec:PosDefHankel} \\[5pt]
\hline
\end{tabular} 
\caption{Summary of the bounds proved on the singular values of matrices with displacement structure. For the singular value bounds 
to be valid for $C_{m,n}$ and $L_n$ mild ``separation conditions" must hold (see Section~\ref{sec:PickCauchyLoewnerMatrices}).
The numbers $\rho_j$ and $C_j$  for $j=1,\ldots,6$ are given explicitly in their corresponding sections.}
\label{tab:SummaryTable}
\end{table} 

Not every matrix with displacement structure is numerically of low rank. For example, the identity matrix is a full rank Toeplitz matrix and the exchange matrix\footnote{The $n\times n$ exchange matrix $X$ is obtained by reversing the order of the rows of the $n\times n$ identity matrix, i.e., $X_{n-j+1,j} = 1$ for $1\leq j\leq n$.} 
is a full rank Hankel matrix. The properties of $A$ and $B$ in~\eqref{eq:SylvesterEquation} are crucial. If $A$ and $B$ are normal matrices, then 
one expects $X$ to be numerically of low rank only if the eigenvalues of $A$ and $B$ are well-separated (see Theorem~\ref{thm:ratio}). If $A$ and $B$ are both not normal, then
as a general rule spectral sets for $A$ and $B$ should be well-separated (see Corollary~\ref{cor:ratioExtension}).

By the Eckart--Young Theorem~\cite[Theorem 2.4.8]{vanLoan}, singular values measure the distance in the spectral norm from $X$ to the set of 
matrices of a given rank, i.e., 
\[
\sigma_j(X) = \min\left\{ \left\|X - Y\right\|_2 : Y\in\mathbb{C}^{m\times n}, \text{ } {\rm rank}(Y) = j-1\right\}. 
\]
For an $0<\epsilon<1$, we say that the {\em $\epsilon$-rank} of a matrix $X$ is $k$ 
if $k$ is the smallest integer such that $\sigma_{k+1}(X) \leq \epsilon \|X\|_2$. That is,
\begin{equation}
{\rm rank}_\epsilon(X) = \min_{k\geq 0} \left\{k : \sigma_{k+1}(X) \leq \epsilon \|X\|_2 \right\}. 
\label{eq:NumericalRank} 
\end{equation} 
Thus, we may approximate $X$ to a precision of $\epsilon \|X\|_2$ by a rank $k={\rm rank}_\epsilon(X)$ matrix. 

An immediate consequence of explicit bounds on the singular values of certain matrices is a bound on the 
$\epsilon$-rank. Table~\ref{tab:SummaryRankTable} summarizes our main upper bounds on the $\epsilon$-rank of 
matrices with displacement structure. 

\begin{table} 
\centering
\begin{tabular}{cccc}
Matrix class & Notation & Upper bound on ${\rm rank}_\epsilon(X)$ & Ref.\\ 
\hline 
\rule{0pt}{1em}Pick & $P_n$ & $2\lceil \log(4b/a)\log(4/\epsilon)/\pi^2\rceil$ & Sec.~\ref{sec:PickMatrices} \\[5pt]
Cauchy & $C_{m,n}$ & $\lceil \log(16\gamma)\log(4/\epsilon)/\pi^2\rceil$ & Sec.~\ref{sec:CauchyMatrices} \\[5pt]
L\"{o}wner &$L_n$ & $2\lceil \log(16\gamma)\log(4/\epsilon)/\pi^2\rceil$ & Sec.~\ref{sec:LoewnerMatrices} \\[5pt]
Krylov, Herm.~arg.~& $K_{m,n}$ & $2\lceil 4\log(8\lfloor n/2\rfloor/\pi)\log(4/\epsilon)/\pi^2\rceil+2$ & Sec.~\ref{sec:VandermondeMatrices}\\[5pt]
Real Vandermonde &$V_{m,n}$ & $2\lceil 4\log(8\lfloor n/2\rfloor/\pi)\log(4/\epsilon)/\pi^2\rceil+2$ & Sec.~\ref{sec:VandermondeMatrices} \\[5pt]
Pos.~semidef.~Hankel &$H_n$ & $2\lceil 2\log(8\lfloor n/2\rfloor/\pi)\log(16/\epsilon)/\pi^2\rceil+2$ & Sec.~\ref{sec:PosDefHankel} \\[5pt]
\hline
\end{tabular} 
\caption{Summary of the upper bounds proved on the $\epsilon$-rank of matrices with displacement structure. For the bounds above
to be valid for $C_{m,n}$ and $L_n$ mild ``separation conditions" must hold (see Section~\ref{sec:PickCauchyLoewnerMatrices}). The number is 
the absolute value of the cross-ratio of $a$, $b$, $c$, and $d$, see~\eqref{eq:gamma}.  The first three rows show an $\epsilon$-rank of at most $\mathcal{O}(\log\gamma\log(1/\epsilon))$ and the last three rows show an $\epsilon$-rank of at most $\mathcal{O}(\log n\log(1/\epsilon))$.}
\label{tab:SummaryRankTable}
\end{table} 

Zolotarev numbers have already proved useful for deriving tight bounds on the condition number of matrices with 
displacement structure~\cite{b26,b46}. For example, the first author proved that a real $n\times n$ positive definite 
Hankel matrix, $H_n$, with $n\geq 3$, is exponentially ill-conditioned~\cite{b46}. That is, 
\[
\kappa_2(H_n)  = \frac{\sigma_1(H_n)}{\sigma_n(H_n)} \geq \frac{\gamma^{n-1}}{16n}, \qquad \gamma \approx 3.210,
\]
and that this bound cannot be improved by more than a factor of $n$ times a modest constant. 
The Hilbert matrix given by $(H_n)_{jk} = 1/(j+k-1)$, for $1\leq j,k\leq n$, is the classic example of an exponentially ill-conditioned positive definite 
Hankel matrix~\cite[eqn.~(3.35)]{wilf}.  Similar exponential ill-conditioning has been shown for certain Krylov matrices and real Vandermonde matrices~\cite{b46}. 

This paper extends the application of Zolotarev numbers to deriving bounds on the singular values of matrices with displacement structure, not just the 
condition number. The bounds we derive are particularly tight for $\sigma_j(X)$, where $j$ is small with respect to $n$. Improved bounds 
on $\sigma_j(X)$ when $j/n\rightarrow c\in (0,1)$ may be possible with the ideas found in~\cite{BB_Gryson}. 
Nevertheless, our interest here is to justify the application of low rank techniques on matrices with displacement structure
by proving that such matrices are often well-approximated by low rank matrices.  The bounds that we derive are sufficient for this purpose. 

For an integer $k$, let $\mathcal{R}_{k,k}$ denote the set of irreducible rational functions of the form $p(x)/q(x)$, where $p$ and $q$ are polynomials of degree at most $k$. 
Given two closed disjoint sets $E,F\subset \mathbb{C}$, the corresponding Zolotarev number, $Z_k(E,F)$, is defined by
\begin{equation} 
    Z_k(E,F) := \inf_{r\in \mathcal{R}_{k,k}}  \frac{{\displaystyle \sup_{z\in E} \left|r(z)\right|}}{{\displaystyle \inf_{z\in F} \left|r(z)\right|}}, 
\label{eq:zolotarev}
\end{equation}
where the infinum is attained for some extremal rational function. As a general rule, the number 
$Z_k(E,F)$ decreases rapidly to zero with $k$ if $E$ and $F$ are sets that are disjoint and well-separated. 
Zolotarev numbers satisfy several immediate properties: for any sets $E$ and $F$ and integers 
$k$, $k_1$, and $k_2$, one has $Z_0(E,F) = 1$, $Z_k(E,F) = Z_k(F,E)$, $Z_{k+1}(E,F)\leq Z_{k}(E,F)$ and 
$Z_{k_1+k_2}(E,F)\leq Z_{k_1}(E,F)Z_{k_2}(E,F)$.  They also satisfy $Z_k(E_1,F_1)\leq Z_k(E_2,F_2)$ if $E_1\subseteq E_2$ and $F_1\subseteq F_2$ as
well as $Z_k(E,F) = Z_k(T(E),T(F))$, where $T$ is any M\"{o}bius transformation~\cite{Akh}. 
As $k\rightarrow\infty$ the value for $Z_k(E,F)$ is known asymptotically to be
\[
\lim_{k\rightarrow\infty} (Z_k(E,F))^{1/k} = \exp\left(-\frac{1}{{\rm cap}(E,F)}\right),
\]
where ${\rm cap}(E,F)$ is the logarithmic capacity of a condenser with plates $E$ and $F$~\cite{Goncar_69_01}. 

To readers that are not familiar with Zolotarev numbers, it may seem that~\eqref{eq:SingularValuesBoundedByZolotarev} trades 
a difficult task of directly bounding the singular values of a matrix $X$ with a more abstract task of understanding the behavior of 
$Z_k(E,F)$. However, Zolotarev numbers have been extensively studied in the literature~\cite{Akh,Goncar_69_01,Zolotarev} and for certain 
sets $E$ and $F$ the extremal rational function is known explicitly~\cite[Sec.~50]{Akh} (see Section~\ref{sec:Zolotarev}). Our 
major challenge for bounding singular values is to carefully select 
sets $E$ and $F$ so that one can use complex 
analysis and M\"{o}bius transformations to convert the associated extremal rational approximation problem in~\eqref{eq:zolotarev} 
into one that has an explicit bound. 

The paper is structured as follows. In Section~\ref{sec:ZolotarevBounds} we prove~\eqref{eq:SingularValuesBoundedByZolotarev}, giving us a bound on the singular values of matrices with displacement structure in terms of Zolotarev numbers.  In Section~\ref{sec:Zolotarev} we derive new sharper bounds on $Z_k([-b,-a],[a,b])$ when $0<a<b<\infty$ by correcting an infinite product formula from Lebedev (see Theorem~\ref{thm:Zproduct} and Corollary~\ref{cor:Zbound}). In Section~\ref{sec:PickCauchyLoewnerMatrices} we derive explicit bounds on the singular values of Pick, Cauchy, and L\"{o}wner matrices.  In Section~\ref{sec:HankelMatrices} we tackle the challenging 
task of showing that all real Vandermonde and positive definite Hankel matrices have rapidly decaying singular values and can be approximated, up to an 
accuracy of $0<\epsilon<1$, by a rank $\mathcal{O}(\log n \log(1/\epsilon))$ matrix. In Appendix~\ref{appendix:correction} we further detail the unfortunate consequences 
of the erroneous infinite product formula from Lebedev and present corrected results. 

\section{The singular values of matrices with displacement structure and Zolotarev numbers}\label{sec:ZolotarevBounds} 
Let $X$ be an $m\times n$ matrix with $m\geq n$ that satisfies~\eqref{eq:SylvesterEquation}.  We show that 
the singular values of $X$ can be bounded from above in terms of Zolotarev numbers.  First, we assume that $A$ and $B$ in~\eqref{eq:SylvesterEquation} are normal 
matrices and later remove this assumption in Corollary~\ref{cor:ratioExtension}. In Theorem~\ref{thm:ratio} the spectrum (set of eigenvalues) 
of $A$ and $B$ is denoted by $\sigma(A)$ and $\sigma(B)$, respectively.\footnote{The statement of Theorem~\ref{thm:ratio} was presented by the first author 
at the Cortona meeting on Structured Numerical Linear Algebra in 2008~\cite{Beckermann_presentation} as well as several other locations. The statement 
has not appeared in a publication by the first (or second) author before. Similar statements based on the presentation have appeared 
in~\cite[Theorem 2.1.1]{Sabino_06_01}, ~\cite[Theorem~4]{Simoncini_16_01}, and most recently~\cite[Theorem 4.2]{Bini_16_01}.}
\begin{theorem}
Let $A\in\mathbb{C}^{m\times m}$ and $B\in\mathbb{C}^{n\times n}$ be normal matrices with $m\geq n$ and let 
$E$ and $F$ be complex sets such that $\sigma(A)\subseteq E$ and $\sigma(B)\subseteq F$. Suppose that 
the matrix $X\in\mathbb{C}^{m\times n}$ satisfies
\[
 AX - XB = MN^*,\qquad M\in\mathbb{C}^{m\times \displacementRank},\quad N\in \mathbb{C}^{n\times \displacementRank},  
\]
where $1\leq \displacementRank\leq n$ is an integer. Then, for $j\geq 1$ the singular values of $X$ satisfy
\[
\sigma_{j+\displacementRank k}(X) \leq Z_k(E,F) \, \sigma_{j}(X),\qquad 1\leq j+\displacementRank k\leq n,
\]
where $Z_k(E,F)$ is the Zolotarev number in~\eqref{eq:zolotarev}.  
\label{thm:ratio}
\end{theorem}
\begin{proof}
Let $p(z)$ and $q(z)$ be polynomials of degree at most $k$. First, we show that 
\begin{equation} 
{\rm rank}( p(A) X q(B) - q(A) X p(B) ) \leq  \displacementRank k,\qquad \displacementRank={\rm rank}(AX-XB).
\label{eq:ADI}
\end{equation}
Suppose that $p(z)=z^s$ and $q(z)=z^t$, where $k \geq s \geq t$, then
\[
\begin{aligned} 
p(A)Xq(B) - q(A)Xp(B) &= A^t\left(A^{s-t}X - XB^{s-t}\right) B^t \\
& = \sum_{j=0}^{s-t-1} A^{t+j}(AX-XB)B^{s-1-j}\\
& = \sum_{j=0}^{s-t-1} \left(A^{t+j}M\right)\left(N^* B^{s-1-j}\right).
\end{aligned} 
\]
In the last sum we have terms of the form $(A^{\ell}  M) (N^* B^{\wp})$, with $0\leq \ell,\wp\leq k-1$. By adding together the 
terms occurring in $p(A) X q(B) - q(A) X p(B)$ for general degree~$k$ polynomials $p$ and $q$, we conclude that there exist coefficients 
$c_{\ell,\wp} \in \mathbb C$ such that
\[
p(A) X q(B) - q(A) X p(B) = \sum_{\ell,\wp=0}^{k-1} c_{\ell,\wp} \left(A^{\ell} M\right)\left(N^* B^{\wp}\right).
\]
This shows that the rank of $p(A) X q(B) - q(A) X p(B)$ is bounded above by $k$ times the number of columns 
of $M$, proving~\eqref{eq:ADI}. 
  
Now, let $r(z)=p(z)/q(z)$, where $p$ and $q$ are polynomials of degree $k$ so that $r(z)$ is the extremal rational function 
for the Zolotarev number in~\eqref{eq:zolotarev}. This means that $p(z)$ and $q(z)$ are not zero on $F$ and $E$, respectively, 
so that $p(B)$ and $q(A)$ are invertible matrices.   From~\eqref{eq:ADI} we know that $\Delta = p(A) X q(B) - q(A) X p(B)$
has rank at most $\displacementRank k$ and hence, the matrix 
\[
Y=-q(A)^{-1} \Delta p(B)^{-1} = X - r(A) X r(B)^{-1} 
\]
is of rank at most $\displacementRank k$. Let $X_j$ be the best rank $j-1$ approximation to $X$ 
in $\|\cdot\|_2$ and let $Y_{j-1} = r(A) X_{j-1} r(B)^{-1}$. Since $Y_{j-1}$ is of rank at most $j-1$, $Y+Y_{j-1}$ is of rank at most 
$j+\displacementRank k-1$. This implies that 
\[
\begin{aligned}
\sigma_{j+\displacementRank k}(X) &\leq \left\| X - Y-Y_{j-1} \right\|_2  \\&= \left\| r(A)\!\left( X - X_{j-1} \right)\! r(B)^{-1} \right\|_2 \\&\leq \left\| r(A) \right\|_2\! \left\| r(B)^{-1} \right\|_2 \sigma_j(X),
\end{aligned} 
\]
where in the last inequality we used the relation $\sigma_j(X) = \|X-X_{j-1}\|_2$. Finally, since $A$ and $B$ are normal we have $\| r(A)\|_2\leq \sup_{z\in\sigma(A)}|r(z)|$ and $\| r(B)^{-1}\|_2\leq \sup_{z\in\sigma(B)}|r(z)^{-1}|$. We conclude by the definition of $r(z)$ that 
\begin{equation}
\frac{\sigma_{j+\displacementRank k}(X)}{\sigma_{j}(X)} \leq \sup_{z\in \sigma(A)} | r(z)| \, \sup_{z\in \sigma(B)} \frac{1}{| r(z)|} =Z_k(\sigma(A),\sigma(B)) \leq Z_k(E,F),
\label{eq:NormalAssumptionEmployed}
\end{equation} 
as required. 
\end{proof}

Theorem~\ref{thm:ratio} shows that if $A$ and $B$ are normal matrices in~\eqref{eq:SylvesterEquation}, then the 
singular values decay at least as fast as $Z_k(\sigma(A),\sigma(B))$ in~\eqref{eq:zolotarev}. In particular,
when $\sigma(A)$ and $\sigma(B)$ are disjoint and well-separated we expect $Z_k(\sigma(A),\sigma(B))$ to decay 
rapidly to zero and hence, so do the singular values of $X$.

For those readers that are familar with the ADI method~\cite{Birkhoff_62_01}, an analogous proof of Theorem~\ref{thm:ratio} is to run 
the ADI method for $k$ steps with shift parameters given by the zeros and poles of the extremal rational 
function for $Z_k(E,F)$.  By doing this one constructs a rank $\displacementRank k$ approximant $X_{\displacementRank k}$ for $X$, which
shows that $\sigma_{1+\displacementRank k}(X) \leq \| X - X_{\displacementRank k}\|_2 \leq Z_{k}(E,F) \sigma_1(X)$.  The connection between 
Zolotarev numbers and the optimal parameter selection for the ADI method has been previously exploited~\cite{Lebedev_77_01}. 
We have presented the above proof here because it does not require knowledge of the ADI method.

For matrices $A$ and $B$ that are not normal, Theorem~\ref{thm:ratio} can be extended by using $K$-spectral sets~\cite{badea_13_01}. Given a matrix $A$, a complex set $E$ is said to be a $K$-spectral set for $A$ if the spectrum $\sigma(A)$ of $A$ is contained in $E$ and the inequality $\|r(A)\|_2 \leq K \|r\|_E$ holds for every bounded rational function on $E$, where $\|r\|_E = \sup_{z\in E} |r(z)|$. Similar extensions have been noted when $B = A^*$ in~\eqref{eq:SylvesterEquation} and the sets $E$ and $F$ are taken to be the fields of values for $A$ and $B$, respectively~\cite{Baker_15_01}. We have the following extension of Theorem~\ref{thm:ratio}:
\begin{corollary}
Suppose that the assumptions of Theorem~\ref{thm:ratio} hold, except that the matrices $A$ and $B$ are not necessarily normal. Also suppose that $E$ and $F$ are $K$-spectral sets for $A$ and $B$ for some fixed constant $K>0$, respectively. 
Then, we have ${\sigma_{j+\displacementRank k}(X)} \leq K^2 Z_k(E,F) \, {\sigma_{j}(X)}$. 
\label{cor:ratioExtension} 
\end{corollary} 
\begin{proof} 
It is only the first inequality in~\eqref{eq:NormalAssumptionEmployed} of the proof of Theorem~\ref{thm:ratio} that requires 
$A$ and $B$ to be normal matrices. When $A$ is not a normal matrix, then the inequality $\|r(A)\|_2\leq \sup_{z\in\sigma(A)} |r(z)|$ may not hold. Instead, we replace it by the $K$-spectral set bound given by $\|r(A)\|_2 \leq K \|r\|_E$. 
Note that since $p(z)$ and $q(x)$ are not zero on $F$ and $E$, respectively, one can show via the Schur decomposition that $p(B)$ and $q(A)$ are invertible matrices.
\end{proof}

Theorem~\ref{thm:ratio} and Corollary~\ref{cor:ratioExtension} provide bounds on the singular values of $X$ in terms of 
Zolotarev numbers. Therefore, to derive analytic bounds on the singular values of matrices with displacement structure, 
we must now calculate explicit bounds on Zolotarev numbers --- a topic that fortunately is extensively studied. 

\section{Zolotarev numbers} \label{sec:Zolotarev}
In this section, we derive explicit lower and upper bounds for the Zolotarev numbers
\[
Z_k := Z_k([-b,-a],[a,b]), \qquad 0<a<b<\infty,
\]
which we use in Sections~\ref{sec:PickCauchyLoewnerMatrices} and~\ref{sec:HankelMatrices}.
The sharpest bounds that we are aware of in the literature take the form
\begin{equation} \label{eq:Gonchar_Braess1}
    \rho^{-2k} \leq Z_k \leq 16 \, \rho^{-2k},
\end{equation}
see~\cite[Theorem~1]{Goncar_69_01} for the lower bound, and~\cite[Eqn.~(2.3)]{braess3} for the upper bound.\footnote{See also \cite[Eqn.~(A1)]{braess2} and~\cite[proof~of Thm. 6.6]{BB_Reichel} for the related problem of minimal Blaschke products, and see~\cite[Theorem~V.5.5]{braess1} for how to deal with rational functions with different degree constraints.} 
There are also bounds obtained directly from an infinite product formula for $\sqrt{Z_k}$~\cite[(1.11)]{Lebedev_77_01}; unfortunately, the original product formula in~\cite[(1.11)]{Lebedev_77_01} contains typos and the erroneous formula has been copied elsewhere, for example,~\cite[(4.1)]{Guttel_14_01},~\cite[(3.17)]{Medovikov_05_01}, and~\cite[Sec.~4]{Oseledets_07_01}.\footnote{As one consequence of the erroneous formula in~\cite[(1.11)]{Lebedev_77_01}, a claimed lower bound in~\cite[(3.17)]{Medovikov_05_01} and~\cite[(1.12)]{Lebedev_77_01} is accidently an upper bound.  Unfortunately, the lower bound in~\cite[(15)]{Oseledets_07_01} also appears to be in error. (See Appendix~\ref{appendix:correction} for more details.)} We prove a corrected infinite product formula in Theorem~\ref{thm:Zproduct} and further discuss the typos in Appendix~\ref{appendix:correction}.

The value of $\rho$ in~\eqref{eq:Gonchar_Braess1} is related to the logarithmic capacity of a condenser with plates $[-b,-a]$ and $[a,b]$:
\begin{equation} \label{eq:Gonchar_Braess2}
    \rho^2=\exp\left(\frac{1}{{\rm cap}([-b,-a],[a,b])}\right),
    \qquad
    \rho=\exp\left(\frac{\pi^2}{2\mu(a/b)}\right),
\end{equation}
where $\mu(\lambda) = \tfrac{\pi}{2} K(\sqrt{1-\lambda^2})/K(\lambda)$ is the Gr\"{o}tzsch ring function, and  $K$ is the complete elliptic integral of the first kind~\cite[(19.2.8)]{NISTHandbook}
\[
       K(\lambda) = \int_0^1 \frac{1}{\sqrt{(1-t^2)(1-\lambda^2t^2)}}dt,\qquad 0\leq \lambda\leq 1.
\]
The bounds in~\eqref{eq:Gonchar_Braess1} are not asymptotically sharp, and in Corollary~\ref{cor:Zbound} we show that the constant of ``$16$" in the upper bound can be replaced by ``$4$". For a proof of this sharper upper bound, we first return to the work of Lebedev~\cite{Lebedev_77_01} and derive a corrected infinite product formula for $Z_k$.

\begin{theorem}\label{thm:Zproduct} 
Let $k\geq 1$ be an integer and $0<a<b<\infty$. Then for $Z_k := Z_k([-b,-a],[a,b])$ we have
\[
Z_k = 4\rho^{-2k}\prod_{\tau=1}^\infty \frac{(1+\rho^{-8\tau k})^4}{(1+\rho^{4k}\rho^{-8\tau k})^4}, \qquad \rho=\exp\left(\frac{\pi^2}{2\mu(a/b)}\right),
\]
where $\mu(\cdot)$ is the Gr\"{o}tzsch ring function.
\end{theorem} 
\begin{proof} 
We start by establishing a product formula for the inverse of $\mu$ that is apparently not widely known. For $\kappa \in (0,1)$ set
$q = {\rm exp}(-2\mu(\kappa))$. Since $\mu(\kappa) = \tfrac{\pi}{2} K\!\left(\sqrt{1-\kappa^2}\right)/K(\kappa)$, we have
that $q = {\rm exp}(-\pi K\!\left(\sqrt{1-\kappa^2}\right)/K(\kappa))$ and from~\cite[(22.2.2)]{NISTHandbook} we obtain
\begin{equation}
\kappa = \left(\frac{\theta_2(0,q)}{\theta_3(0,q)}\right)^{2} = 4\sqrt{q}\prod_{\tau=1}^\infty\frac{(1+q^{2\tau})^4}{(1+q^{2\tau-1})^4}, \quad q = q(\kappa)={\rm exp}\!\left(\!-2\mu(\kappa)\right).
\label{eq:rhoProduct} 
\end{equation} 
Here, $\theta_2(z,q)$ and $\theta_3(z,q)$ are the classical theta functions~\cite[(20.2.2) \& (20.2.3)]{NISTHandbook} and the second equality above is derived from the infinite product formula in~\cite[(20.4.3) \& (20.4.4)]{NISTHandbook}.

In order to deduce an explicit product formula for $Z_k$, we first note that the value of $2\sqrt{Z_{k}}/(1+Z_{k})$ is extensively reviewed by Akhiezer,\footnote{There is a typo in~\cite[Tab.~1 \& 2, p.~150, No.~7 \& 8]{Akh}. There should be no prime on $\lambda_1$.} see~\cite[Sec.~51]{Akh},~\cite[Tab.~1 \& 2, p.~150, No.~7 \& 8]{Akh}, and~\cite[Tab.~XXIII]{Akh}.   This value is equal to~\cite[p.~149]{Akh}, for some $\lambda_k\in(0,1)$,
 \[
\frac{2\sqrt{Z_k}}{1+Z_k} = \frac{1-\lambda_k}{1+\lambda_k}, \qquad k\mu(\lambda_k) = \mu(a/b).
\]
Here, there is a unique $\lambda_k\in(0,1)$ since the Gr\"{o}tzsch ring function $\mu:[0,1]\rightarrow[0,\infty]$ is a strictly decreasing bijection.  Next, we 
recall that Gauss' transformation~\cite[Tab.~XXI]{Akh} and Landen's transformation~\cite[Tab.~XX]{Akh} are given by
\begin{equation}
\mu\!\left(\frac{2\sqrt{\lambda}}{1+\lambda}\right) = \frac{\mu(\lambda)}{2}, \qquad \mu\!\left(\frac{1-\lambda}{1+\lambda}\right) = 2 \mu\!\left(\!\sqrt{1-\lambda^2}\right),\qquad \lambda\in(0,1),
\label{eq:GaussTransform}
\end{equation}
   from which we conclude that
   \begin{equation} \label{eq:LandenTransformation2}
       \mu(Z_k) = 2 \mu\!\left(\frac{2\sqrt{Z_k}}{1+Z_k}\right) = 4 \mu\!\left(\sqrt{1-\lambda_k^2}\right) = \frac{\pi^2}{\mu(\lambda_k)} = \frac{\pi^2 k}{\mu(a/b)}.
   \end{equation}
Therefore, from~\eqref{eq:LandenTransformation2} we have 
\[
q = q(Z_k) = e^{-2\mu(Z_k)} = {\rm exp}\left(-2k\frac{\pi^2}{\mu(a/b)}\right) = \rho^{-4k},
\]
where $\rho$ is given in~\eqref{eq:Gonchar_Braess2}. The infinite product formula for $Z_k$ follows by setting $\kappa = Z_k$ and $q = \rho^{-4k}$ in~\eqref{eq:rhoProduct}. 
\end{proof} 

The infinite product in Theorem~\ref{thm:Zproduct} can be estimated by observing that 
$(1+\rho^{-4k} \rho^{-8\tau k})^2 \leq (1+ \rho^{-8\tau k})^2 \leq (1+\rho^{-4k} \rho^{-8\tau k}) (1+\rho^{4k} \rho^{-8\tau k})$ for all $\tau\geq 1$. 
This leads to the following simple upper and lower bounds which are sufficient for the purpose of our paper.

\begin{corollary} 
Let $k\geq 1$ be an integer and $0<a<b<\infty$. Then for $Z_k := Z_k([-b,-a],[a,b])$ we have
\[
\frac{4\rho^{-2k}} {(1+\rho^{-4k})^4} \leq Z_k \leq \frac{4\rho^{-2k}}{ (1+\rho^{-4k})^2} \leq 4\rho^{-2k}, \qquad \rho=\exp\left(\frac{\pi^2}{2\mu(a/b)}\right),
\]
where $\mu(\cdot)$ is the Gr\"{o}tzsch ring function.
\label{cor:Zbound}
\end{corollary} 

Corollary~\ref{cor:Zbound} shows that $Z_k\leq 4\rho^{-2k}$ is an asymptotically sharp upper bound in the sense that the geometric decay 
rate and the constant ``4" cannot be improved if one hopes for the bound to hold for all $k$. However, this does not necessarily imply that our derived 
singular value inequalities are asymptotically sharp. On the contrary, they are usually not. For asymptotically sharp singular value bounds, we expect that one must consider 
discrete Zolotarev numbers, i.e., $Z_k(\sigma(A),\sigma(B))$ in Theorem~\ref{thm:ratio}, which are more subtle to bound and are outside the scope of this paper.  

We often prefer the following slightly weaker bound that does not contain the Gr\"{o}tzsch ring function: 
\[
Z_k([-b,-a],[a,b]) \leq 4\left[\exp\left(\frac{\pi^2}{2\log(4b/a)}\right)\right]^{-2k},\qquad 0<a<b<\infty,
\]
which is obtained by using the bound $\mu(\lambda)\leq \log((2(1+\sqrt{1-\lambda^2}))/\lambda \leq \log(4/\lambda)$, see~\cite[(19.9.5)]{NISTHandbook}.   This makes our final bounds on the singular values 
and $\epsilon$-rank of matrices with displacement rank more intuitive to those readers that are less 
familiar with the Gr\"{o}tzsch ring function.

Later, in Section~\ref{sec:HankelMatrices} we will need to use properties of an extremal rational function for $Z_k=Z_k([-b,-a],[a,b])$ and we proof them now. 
Zolotarev~\cite{Zolotarev} studied the value $Z_k$ and gave an explicit expression for the extremal function for $Z_k$ (see~\eqref{eq:zolotarev}) by 
showing an equivalence to the problem of best rational approximation of the sign function on $[-b,-a]\cup[a,b]$.  We now repeat this to derive the desired properties 
of the extremal rational function. 
\begin{theorem} 
Let $k\geq 1$ be an integer and $0<a<b<\infty$.  There exists an extremal function $R \in \mathcal{R}_{k,k}$ for $Z_k=Z_k([-b,-a],[a,b])$ such that 
\begin{enumerate}[label=(\alph*)]
\item For $z\in [-b,-a]$, we have $-\sqrt{Z_k}\leq R(z) \leq \sqrt{Z_k}$, 
\item For $z\in\mathbb{C}$, we have $R(-z)=1/R(z)$, and 
\item For $z\in \mathbb{R}$, we have $|R(iz)|=1$. 
\end{enumerate} 
\label{thm:ZextremalRational} 
\end{theorem}
\begin{proof}
We give an explicit expression for an extremal function for $Z_k$ by deriving it from the best rational approximation of the sign function on $[-b,-a]\cup [a,b]$. 
 According to \cite[Sec.~50 \& 51, p.~144, line 6]{Akh} we have
   \[
       \inf_{r\in {\mathcal R}_{k,k}} \sup_{z\in [-b,-a]\cup[a,b]}\left|{\rm sgn}(z) - r(z)\right| = \frac{2\sqrt{Z_k}}{1+Z_k},\qquad {\rm sgn}(z) = \begin{cases}1,&z\in [a,b],\\ -1,&z\in [-b,-a], \end{cases}
   \] 
   where the infimum is attained by the rational function~\cite[Sec.~51, Tab.~2, No.~7 \& 8]{Akh}
   \begin{equation}
       \tilde{r}(z) = M z \frac{\prod_{j=1}^{\lfloor (k-1)/2\rfloor} z^2+c_{2j}}{\prod_{j=1}^{\lfloor k/2\rfloor} z^2+c_{2j-1}}, \qquad c_j = a^2\frac{{\rm sn}^2(jK(\kappa)/k;\kappa)}{1-{\rm sn}^2(jK(\kappa)/k;\kappa)}.
       \label{eq:Explicit}
   \end{equation}
   Here, $M$ is a real constant selected so that ${\rm sgn}(z) - \tilde{r}(z)$ equioscillates on $[-b,-a]\cup[a,b]$, $\kappa = \sqrt{1 - (a/b)^2}$, and ${\rm sn}(\cdot)$ is the first Jacobian elliptic function.
   \begin{figure}
   \begin{minipage}{.49\textwidth}
   \begin{overpic}[width=\textwidth]{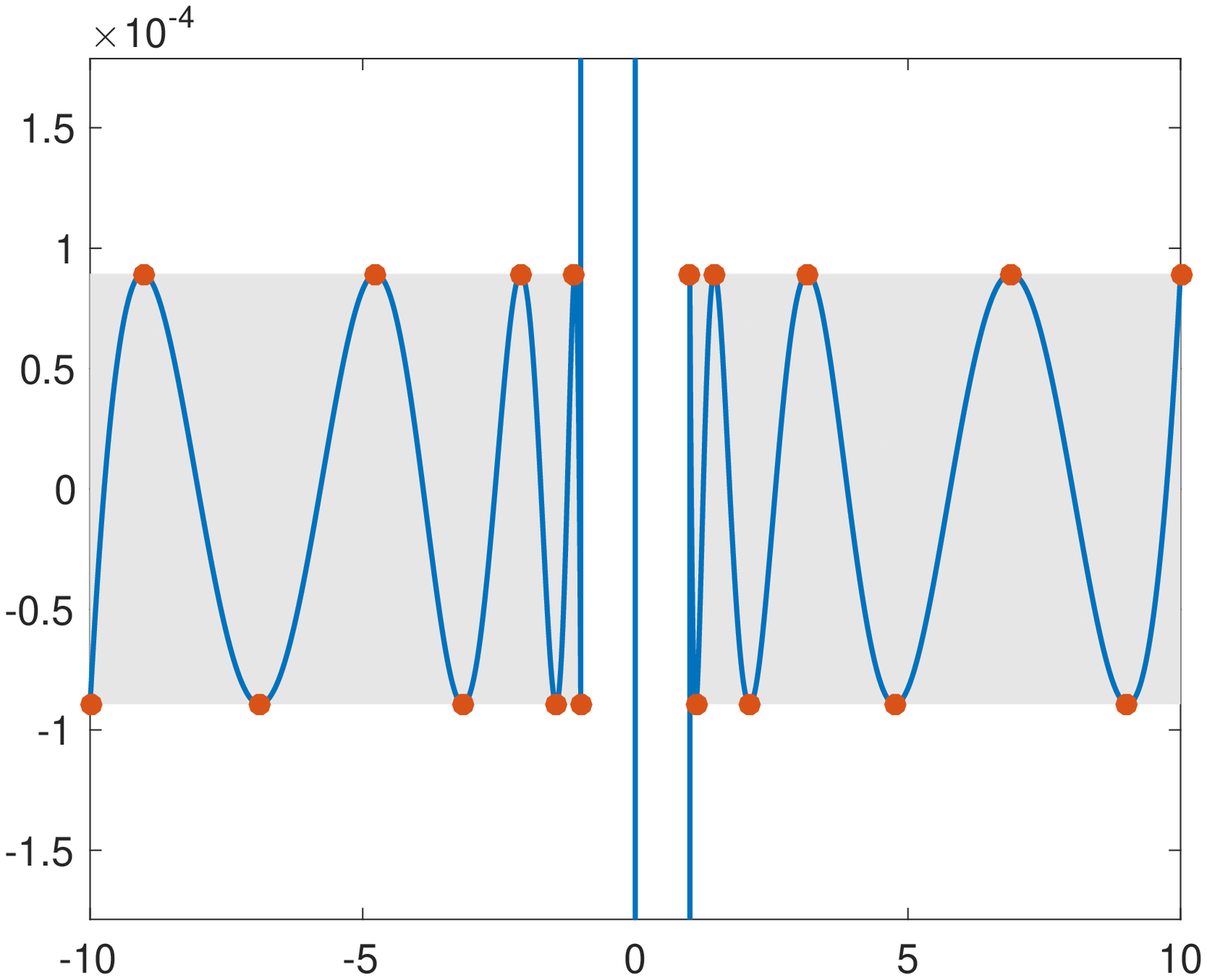}
   \put(50,0) {$x$}
   \end{overpic}
   \end{minipage}
   \begin{minipage}{.49\textwidth}
   \begin{overpic}[width=\textwidth]{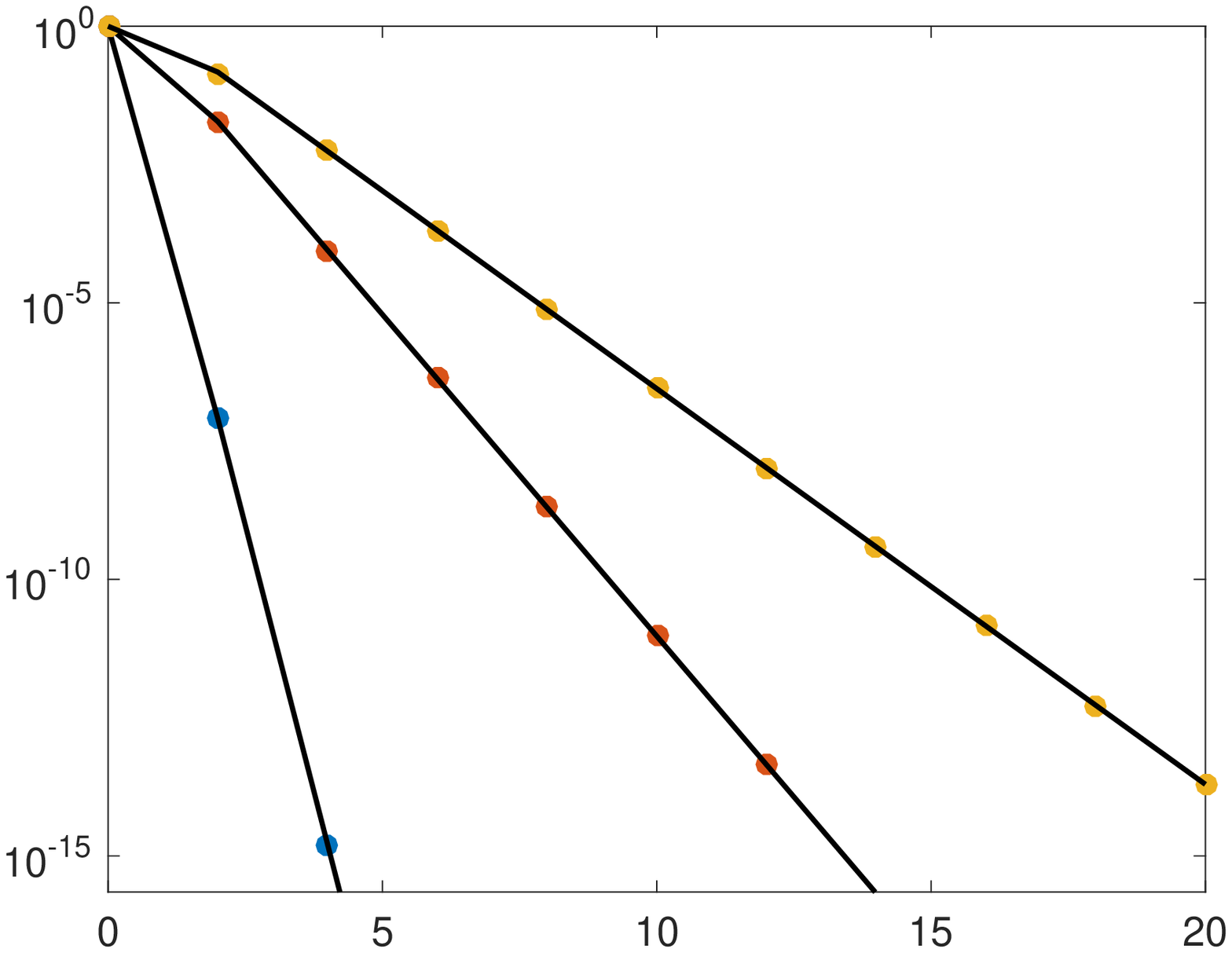}
   \put(50,0) {$k$}
   \put(25,30) {\rotatebox{-73}{$b/a = 1.1$}}
   \put(52,29) {\rotatebox{-47}{$b/a = 10$}}
   \put(67,35) {\rotatebox{-36}{$b/a = 100$}}
   \end{overpic}
   \end{minipage}
   \caption{Zolotarev's rational approximations. Left: The error between the sign function on $[-10,-1]\cup[1,10]$ and its best rational ${\mathcal R}_{8,8}$ approximation on the domain $[-10,10]$. The error equioscillates $9$ times in the interval $[-10,-1]$ and $[1,10]$ (see red dots), verifying its optimality~\cite[p.~149]{Akh}.  Right: The upper bound (black line) on $Z_k([-b,-a],[a,b])$ (colored dots) in Corollary~\ref{cor:Zbound} for $0\leq k \leq 20$, with $b/a = 1.1$ (blue), $10$ (red), $100$ (yellow).}
   \label{fig:ZolotarevBounds}
   \end{figure}
   Figure~\ref{fig:ZolotarevBounds} (left) shows the error between the sign function on $[-10,-1]\cup[1,10]$ and its best ${\mathcal R}_{8,8}$ rational approximation, which equioscillates $9$ times on $[-10,-1]$ and $[1,10]$ to confirm its optimality.
   
In order to construct an extremal function for $Z_k$ with the required properties, we observe from \eqref{eq:Explicit} that $M$ and $c_1,\ldots, c_{k-1}$ are real, and thus
   \begin{itemize}
    \item $\tilde{r}(z)$ is real-valued for $z\in \mathbb R$,
    \item $\tilde{r}(iz)$ is purely imaginary for $z\in\mathbb{R}$, and
    \item $\tilde{r}(z)$ is an odd function on $\mathbb{R}$, i.e., $\tilde{r}(z) = -\tilde{r}(-z)$ for $z\in\mathbb{R}$.
   \end{itemize}
   As a consequence, the rational function given by
\[
           R(z) = \frac{1 + \frac{1+Z_k}{1-Z_k} \tilde r(z)}{1 - \frac{1+Z_k}{1-Z_k} \tilde r(z)} \in\mathcal{R}_{k,k}
\]
   is real-valued for $z\in \mathbb R$ with $R(-z)=1/R(z)$, and of modulus $1$ on the imaginary axis. Finally, as $\tilde r(z)$ takes values in the interval
   \[
       \left[-1-\frac{2\sqrt{Z_k}}{1+Z_k}, -1+\frac{2\sqrt{Z_k}}{1+Z_k}\right]
       = \left[\frac{-(1+\sqrt{Z_k})^2}{1+Z_k},
       \frac{-(1-\sqrt{Z_k})^2}{1+Z_k}\right]
   \]
   for $z\in [-b,-a]$, we have
   \[
       \frac{1+Z_k}{1-Z_k} \tilde r(z)
       \in
       \left[-\frac{1+\sqrt{Z_k}}{1-\sqrt{Z_k}},
       -\frac{1-\sqrt{Z_k}}{1+\sqrt{Z_k}}
       \right],
   \]
   implying that $-\sqrt{Z_k}\leq R(z)\leq \sqrt{Z_k}$ for $z\in [-b,-a]$. Hence, using $R(-z)=1/R(z)$ we have 
   \[
       \frac{{\displaystyle \sup_{z\in E} \left|R(z)\right|}}{{\displaystyle \inf_{z\in F} \left|R(z)\right|}} \leq Z_k = \inf_{r\in \mathcal{R}_{k,k}}  \frac{{\displaystyle \sup_{z\in E} \left|r(z)\right|}}{{\displaystyle \inf_{z\in F} \left|r(z)\right|}},
   \]
   showing that $R$ is extremal for $Z_k([-b,-a],[a,b])$, as required.
\end{proof} 

Figure~\ref{fig:ZolotarevBounds} (right) demonstrates the upper bound in Corollary~\ref{cor:Zbound} when $b/a = 1.1$, $10$, $100$. 
In Section~\ref{sec:PickCauchyLoewnerMatrices} we combine our upper bound on the singular values in Theorem~\ref{thm:ratio} with 
our upper bound on Zolotarev numbers to derive explicit bounds on the singular values of certain Pick, Cauchy, and 
L\"{o}wner matrices. 
\section{The decay of the singular values of Pick, Cauchy, and L\"{o}wner matrices}\label{sec:PickCauchyLoewnerMatrices}
In this section we bound the singular values of Pick (see Section~\ref{sec:PickMatrices}), Cauchy (see Section~\ref{sec:CauchyMatrices}), 
and L\"{o}wner (see Section~\ref{sec:LoewnerMatrices}) matrices. In view of Theorem~\ref{thm:ratio} and Corollary~\ref{cor:Zbound}, 
our first idea is to construct matrices $A$ and $B$ so that the rank of $AX-XB$ is small with the additional hope that $\sigma(A)$ and $\sigma(B)$ are 
contained in real and disjoint intervals.   For the three classes of matrices in this section, this first idea works out under mild ``separation conditions''. In
Section~\ref{sec:HankelMatrices} the more challenging cases of Krylov, real Vandermonde, and real positive definite Hankel matrices are considered.

\subsection{Pick matrices}\label{sec:PickMatrices} 
An $n\times n$ matrix $P_n$ is called a Pick matrix if there exists a vector $\underline{s}=(s_1,\ldots,s_n)^T\in\mathbb{C}^{n\times 1}$, 
and a collection of real numbers $x_1<\cdots <x_n$ from an interval $[a,b]$ with $0<a<b<\infty$ such that 
\begin{equation}
(P_n)_{jk} = \frac{s_j+s_k}{x_j+x_k}, \qquad 1\leq j,k\leq n. 
\label{eq:PickMatrix}
\end{equation} 
All Pick matrices satisfy the following Sylvester matrix equation: 
\begin{equation}
D_{\underline{x}}P_n - P_n(-D_{\underline{x}})  = \underline{s}\,\underline{e}^T+\underline{e}\,\underline{s}^T,\qquad 
D_{\underline{x}} = {\rm diag}\left(x_1,\ldots,x_n\right),
\label{eq:PickSylvester}
\end{equation} 
where $\underline{e} = (1,\ldots,1)^T$.  Since diagonal matrices are normal 
matrices and in this case the spectrum of $D_{\underline{x}}$ is contained in $[a,b]$, we have the following bounds on
the singular values of $P_n$.
\begin{corollary} 
Let $P_n$ be the $n\times n$ Pick matrix in~\eqref{eq:PickMatrix}. Then, for $j\geq 1$ we have
\[
\sigma_{j+2k}(P_n) \leq  4\left[\exp\left(\frac{\pi^2}{2\mu(a/b)}\right)\right]^{-2k} \sigma_j(P_n),\qquad 1\leq j+2k\leq n,
\]
where $\mu(\lambda)$ is the Gr\"{o}tzsch ring function (see Section~\ref{sec:Zolotarev}). The bound remains valid, but is slightly 
weaken, if $\mu(a/b)$ is replaced by $\log(4b/a)$.  
\label{cor:PickMatrices}
\end{corollary} 
\begin{proof} 
From~\eqref{eq:PickSylvester}, we know that $A = D_{\underline{x}}$, $B=-A$, $\displacementRank = 2$, $E = [a,b]$, and $F=[-b,-a]$ in Theorem~\ref{thm:ratio}. 
Therefore, for $j\geq 1$ we have $\sigma_{j+2k}(P_n) \leq Z_k(E,F)\sigma_j(P_n)$, $1\leq j+2k\leq n$. The result follows 
from the upper bound in Corollary~\ref{cor:Zbound}.
\end{proof}

There are two important consequences of Corollary~\ref{cor:PickMatrices}: (1) Pick matrices are usually ill-conditioned unless $b/a$ is large 
and/or $n$ is small and (2) All Pick matrices can be approximated, up to an accuracy of $\epsilon\|X\|_2$ with $0<\epsilon<1$, 
by a rank $\mathcal{O}(\log(b/a)\log(1/\epsilon))$ matrix.  More precisely, for any Pick matrix in~\eqref{eq:PickMatrix} we have 
\begin{equation} 
\kappa_2(P_n) = \frac{\sigma_1(P_n)}{\sigma_n(P_n)} \geq \frac{1}{4}\left[\exp\left(\frac{\pi^2}{2\log(4b/a)}\right)\right]^{2\lceil \tfrac{n}{2}-1\rceil},
\label{eq:PickConditionNumber} 
\end{equation} 
where for an even integer $n$ we used $\sigma_1(P_n)/\sigma_{n}(P_n) \geq \sigma_1(P_n)/\sigma_{n-1}(P_n)$. 
Moreover, by setting $k$ to be the smallest integer so that $\sigma_{1+2k}(P_n)\leq \epsilon \sigma_1(P_n)$, we find the following bound on  
the $\epsilon$-rank of $P_n$ (see~\eqref{eq:NumericalRank}):
\begin{equation}
{\rm rank}_\epsilon(P_n) \leq 2\bigg\lceil\frac{\log(4b/a)\log(4/\epsilon)}{\pi^2}\bigg\rceil.
\label{eq:PickNumericalRank}
\end{equation} 

In both~\eqref{eq:PickConditionNumber} and~\eqref{eq:PickNumericalRank}, the bound can be slightly improved 
by replacing the $\log(4b/a)$ term by $\mu(a/b)$.  Previously, bounds on the minimum and maximum singular values 
of Pick matrices were derived under the additional assumption that $P_n$ is a positive definite matrix~\cite{fasino}. 

Figure~\ref{fig:PickCauchyMatrices} (left) demonstrates the bound in Corollary~\ref{cor:PickMatrices} on three $100\times 100$ 
Pick matrices. The black line bounding the singular values has a stepping behavior because the inequality in 
Corollary~\ref{cor:PickMatrices} for $j = 1$ only bounds odd indexed singular values of $P_n$ and to bound $\sigma_{2k}(P_n)$ we use the trivial 
inequality $\sigma_{2k}(P_n)\leq \sigma_{2k-1}(P_n)$.  At this time we can offer little insight into why the singular values of the 
tested Pick matrices also have a similar stepping behavior.
 
\begin{figure} 
\begin{minipage}{.49\textwidth} 
\begin{overpic}[width=\textwidth]{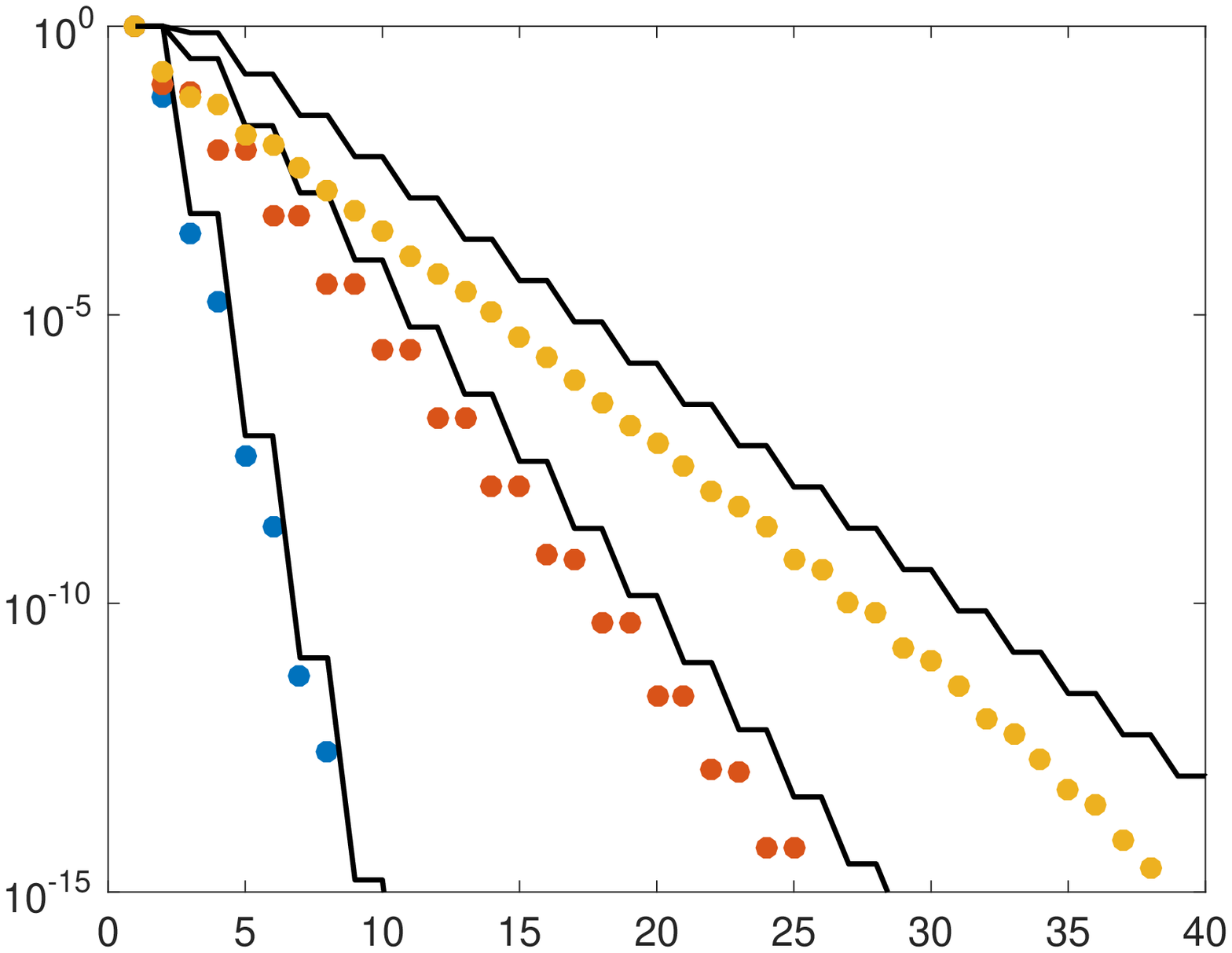}
\put(35,0) {$\sigma_j(P_n)/\sigma_1(P_n)$}
\put(28,30) {\rotatebox{-80}{$\frac{b}{a}=1.1$}}
\put(57,26) {\rotatebox{-50}{$\frac{b}{a}=10$}}
\put(70,35) {\rotatebox{-35}{$\frac{b}{a}=100$}}
\end{overpic} 
\end{minipage} 
\begin{minipage}{.49\textwidth} 
\begin{overpic}[width=\textwidth]{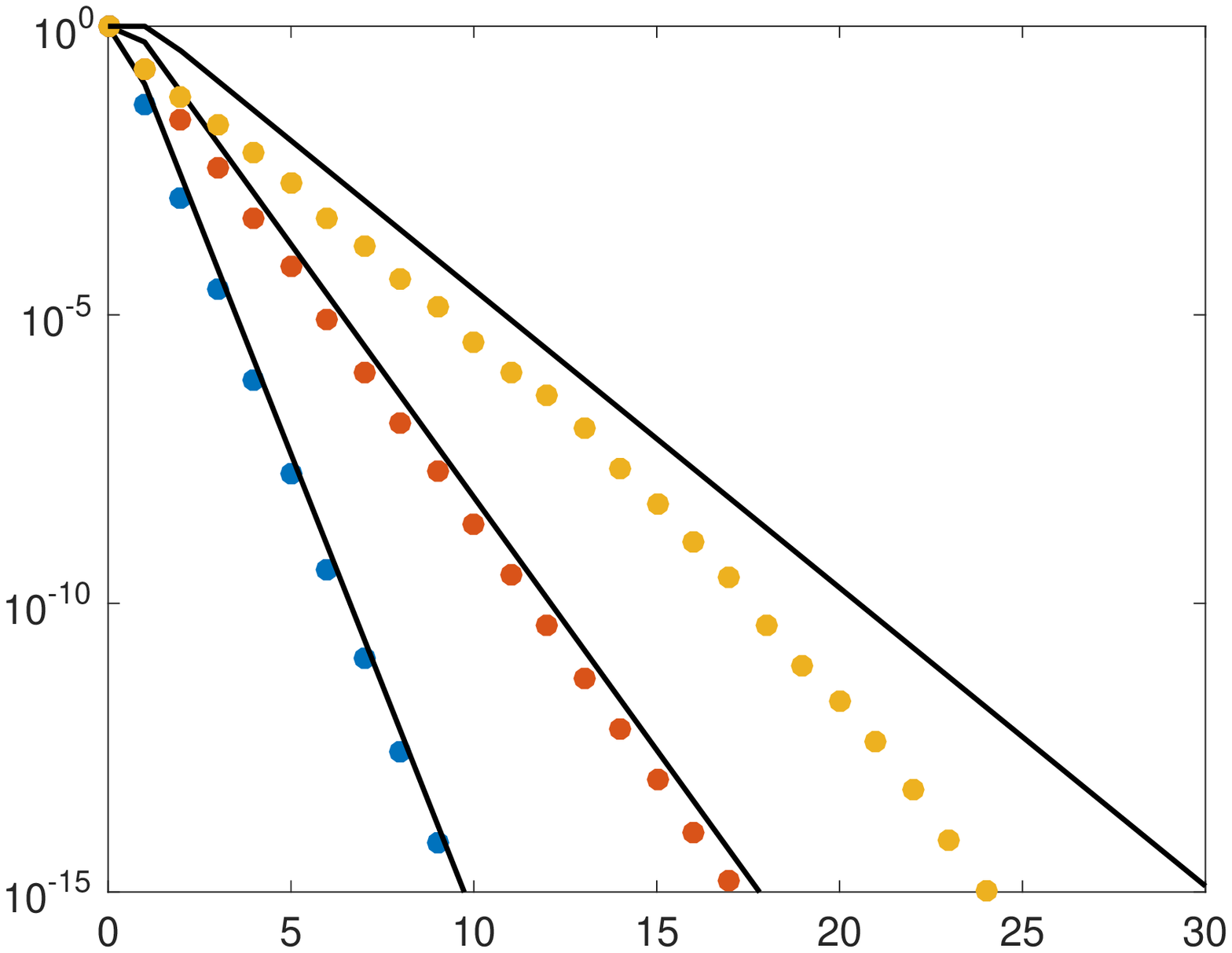}
\put(33,0) {$\sigma_j(C_{n,n})/\sigma_1(C_{n,n})$}
\put(31,30) {\rotatebox{-65}{$\gamma\approx1.492$}}
\put(42,34) {\rotatebox{-53}{$\gamma\approx8.841$}}
\put(65,33) {\rotatebox{-37}{$\gamma\approx251.46$}}
\end{overpic} 
\end{minipage} 
\caption{Left: The scaled singular values of $100\times 100$ Pick matrices (colored dots) and the bound in 
Corollary~\ref{cor:PickMatrices} (black line) for $b/a = 1.1$ (blue), $10$ (red), $100$ (yellow). In~\eqref{eq:PickMatrix}, 
$\underline{x}$ is a vector of equally spaced points in $[a,b]$ and $\underline{s}$ is a random vector with 
independent standard Gaussian entries. 
Right: The scaled singular values of $100\times 100$ Cauchy matrices (colored dots) and the bound in Corollary~\ref{cor:CauchyMatrices}  (black line) for $\gamma = 1.1, 10,100$. In~\eqref{eq:CauchyMatrix}, $\underline{x}$ is a vector of Chebyshev nodes from $[-8.5,-2]$ (blue), $[-100,-3]$ (red), and $[-101,2.8]$ (yellow), respectively, 
$\underline{y}$ is a vector of Chebyshev nodes from $[3,10]$ (blue), $[3,100]$ (red), and $[3,100]$ (yellow), respectively, and $\underline{s}$ and $\underline{t}$ are 
random vectors with independent standard Gaussian entries. The decay rate depends on the cross-ratio of the endpoints of the intervals.}
\label{fig:PickCauchyMatrices} 
\end{figure} 

\subsection{Cauchy matrices}\label{sec:CauchyMatrices}
An $m\times n$ matrix $C_{m,n}$ with $m\geq n$ is called a (generalized) Cauchy matrix if there exists vectors $\underline{s}\in\mathbb{C}^{m\times 1}$ and $\underline{t}\in\mathbb{C}^{n\times 1}$, points $x_1<\cdots <x_m$ on an interval $[a,b]$ with $-\infty<a<b<\infty$, and points $y_1<\cdots<y_n$ (all distinct from $x_1,\ldots,x_m$) in an interval 
$[c,d]$ with $-\infty<c<d<\infty$ such that
\begin{equation} 
(C_{m,n})_{jk} = \frac{s_jt_k}{x_j-y_k}, \qquad 1\leq j\leq m, \quad 1\leq k\leq n. 
\label{eq:CauchyMatrix} 
\end{equation}
Generalized Cauchy matrices satisfy the following Sylvester matrix equation: 
\begin{equation}
D_{\underline{x}}C_{m,n} -  C_{m,n} D_{\underline{y}}  = \underline{s}\,\underline{t}^T, 
\label{eq:CauchySylvester}
\end{equation} 
where $D_{\underline{x}} = {\rm diag}\left(x_1,\ldots,x_m\right)$ and $D_{\underline{y}} = {\rm diag}\left(y_1,\ldots,y_n\right)$.  

If we make the further assumption that either $b<c$ or $d<a$ so that the intervals $[a,b]$ and $[c,d]$ are disjoint, then we can 
bound the singular values of $C_{m,n}$. This ``separation condition" is an extra assumption on Cauchy matrices that simplifies the analysis. 
If the intervals $[a,b]$ and $[c,d]$ overlapped, then one would have to consider discrete Zolotarev numbers to estimate the singular values and we want to avoid this in this paper. 
\begin{corollary} 
Let $C_{m,n}$ be an $m\times n$ Cauchy matrix in~\eqref{eq:CauchyMatrix} with $m\geq n$ and either $b<c$ or $d<a$. Then, 
\[
\sigma_{j+k}(C_{m,n}) \leq 4\left[\exp\left(\frac{\pi^2}{4\mu(1/\sqrt{\gamma})}\right)\right]^{-2k}\sigma_j(C_{m,n}), \qquad 1\leq j+k\leq n, 
\]
where $\gamma$ is the absolute value of the cross-ratio\footnote{Given four collinear points $a$, $b$, $c$, and $d$ the cross-ratio is given by $(c-a)(d-b)/((c-b)(d-a))$.} of $a$, $b$, $c$, and $d$. If $a = c$ and $b = d$, then $2\mu(1/\sqrt{\gamma}) = \mu(a/b)$. The bound remains valid, but is slightly 
weaken, if $4\mu(1/\sqrt{\gamma})$ is replaced by $2\log(16\gamma)$.  
\label{cor:CauchyMatrices}
\end{corollary} 
\begin{proof} 
From~\eqref{eq:CauchySylvester}, we know that $A = D_{\underline{x}}$, $B = D_{\underline{y}}$, $\displacementRank = 1$, 
$E = [a,b]$, and $F=[c,d]$ in Theorem~\ref{thm:ratio}. Therefore, we conclude that 
$\sigma_{j+k}(C_{m,n}) \leq Z_k(E,F)\sigma_j(C_{m,n})$ for $1\leq j+k\leq n$. 

The value of $Z_k(E,F)$ is invariant under M\"{o}bius transformations. That is, if $T(z) = (a_1z+a_2)/(a_3z+a_4)$ 
is a M\"{o}bius transformation, then $Z_k(E,F)$ and $Z_k(T(E),T(F))$ are equal. Therefore, we can transplant $[a,b]\cup[c,d]$ onto 
$[-\alpha,-1]\cup[1,\alpha]$ for some $\alpha>1$ using a M\"{o}bius transformation. If $b<c$, then the transformation satisfies 
$T(a) = -\alpha$, $T(b) = -1$, $T(c)=1$, $T(d)=\alpha$.   Since $T$ is a M\"{o}bius transformation the cross-ratio of the four collinear points $a$, $b$, $c$, and $d$ equals the cross-ratio of $T(a)$, $T(b)$, $T(c)$, and $T(d)$. 
Hence, if $b<c$ or $d<a$ then we know that $\alpha$ must satisfy
\[
\frac{|c-a| |d-b|}{|c-b| |d-a|} = \frac{(\alpha+1)^2}{4\alpha}.
\]
Therefore, by solving the quadratic and noting that $\alpha> 1$ we have
\begin{equation}
\alpha = -1 + 2\gamma + 2\sqrt{\gamma^2 - \gamma}, \qquad \gamma = \frac{|c-a| |d-b|}{|c-b| |d-a|}.
\label{eq:gamma}
\end{equation} 
From Corollary~\ref{cor:Zbound}, we conclude that 
\[
\sigma_{j+k}(C_{m,n}) \leq 4\left[\exp\left(\frac{\pi^2}{2\mu(1/\alpha)}\right)\right]^{-2k}\sigma_j(C_{m,n}), \qquad 1\leq j+k\leq n.
\]
By Gauss' transformation in~\eqref{eq:GaussTransform}, we note that $\mu(1/\alpha)=2\mu( 1/\sqrt{\gamma} ) \leq 2 \log(4 \sqrt{\gamma}) = \log(16 \gamma)$ and 
the result follows. 
\end{proof}

It is interesting to observe that the decay rate of the singular values of Cauchy matrices only depends on the absolute value of the cross-ratio of $a$, $b$, $c$, and $d$. Hence, the ``separation" of two real intervals $[a,b]$ and $[c,d]$ for the purposes of singular value estimates is measured in terms of the cross-ratio of $a$, $b$, $c$, and $d$, not the 
separation distance $\max(c-b,a-d)$. 

Corollary~\ref{cor:CauchyMatrices} shows that the Cauchy matrix in~\eqref{eq:CauchyMatrix} (when $b<c$ or $d<a$)
has an $\epsilon$-rank of at most
\[
{\rm rank}_\epsilon(C_{m,n}) \leq \bigg\lceil\frac{2\mu(1/\sqrt{\gamma})\log(4/\epsilon)}{\pi^2}\bigg\rceil \leq \bigg\lceil\frac{\log(16\gamma)\log(4/\epsilon)}{\pi^2}\bigg\rceil,
\]
where $\gamma$ is absolute value of the cross-ratio of $a$, $b$, $c$, and $d$.
Bounds on the numerical rank of the Cauchy matrix have also been obtained via the Cauchy function, i.e., $1/(x+y)$ on $[a,b]\times [c,d]$, by exploiting 
the hierarchical low rank structure of $C_{m,n}$~\cite{Grasedyck_04_01} (for more details, see~\cite[Chapter 3]{Townsend_14_01}).  Furthermore, when $m = n$ (and $b<c$ or $d<a$) we have a lower bound on the condition number of $C_{n,n}$:
\[
\kappa_2(C_{n,n}) = \frac{\sigma_1(C_{n,n})}{\sigma_n(C_{n,n})} \geq \frac{1}{4}\left[\exp\left(\frac{\pi^2}{2\log(16\gamma)}\right)\right]^{2(n-1)},\qquad \gamma = \frac{|c-a| |d-b|}{|c-b| |d-a|}.
\]

Corollary~\ref{cor:CauchyMatrices} also includes the important Hilbert matrix, i.e., $(H_n)_{jk} = 1/(j+k-1)$ for $1\leq j,k\leq n$. 
By setting $x_j = j-1/2$, $y_j = -k+1/2$, and $\underline{s} = \underline{r} = (1,\ldots,1)^T$, the matrix in~\eqref{eq:CauchyMatrix}
is the Hilbert matrix. In particular, Corollary~\ref{cor:CauchyMatrices}  with $[a,b] = [-n+1/2,-1/2]$ and $[c,d] = [1/2,n-1/2]$ shows that
\begin{equation}
\sigma_{k+1}(H_n) \leq 4\left[\exp\left(\frac{\pi^2}{2\log(8n-4)}\right)\right]^{-2k}\sigma_1(H_n),\qquad 1\leq k\leq n-1. 
\label{eq:HilbertMatrix}
\end{equation} 
Therefore, the Hilbert matrix can be well-approximated by a low rank matrix and has exponentially decaying singular values.\footnote{More generally, skeleton decompositions can be used to show that the Hilbert kernel of $f(x,y) = 1/(x+y)$ on $[a,b]\times [a,b]$ with $0<a<b<\infty$ has exponentially decaying singular values~\cite{Oseledets_07_01}. Even though there is an error in~\cite[Sec.~4]{Oseledets_07_01} in the infinite product formula and the stated lower bound (see Appendix~\ref{appendix:correction}), we believe the proved upper bound in~\cite[(15)]{Oseledets_07_01} is correct.} In particular, 
it has an $\epsilon$-rank of at most $\lceil \log(8n-4)\log(4/\epsilon)/\pi^2 \rceil$.  The Hilbert matrix is an example of a real positive definite 
Hankel matrix and in Section~\ref{sec:HankelMatrices} we show that bounds similar to~\eqref{eq:HilbertMatrix} hold for the singular values 
of all such matrices. 
 
Figure~\ref{fig:PickCauchyMatrices} (right) demonstrates the bound in Corollary~\ref{cor:CauchyMatrices}  on three $n\times n$ 
Cauchy matrices, where $n = 100$. In practice, the derived bound is relatively tight for singular values $\sigma_j(C_{m,n})$ 
when $j$ is small with respect the $n$. 
 
\subsection{L\"{o}wner matrices}\label{sec:LoewnerMatrices}
An $n\times n$ matrix $L_n$ is called a L\"{o}wner matrix if there exist vectors $\underline{r},\underline{s}\in\mathbb{C}^{n\times 1}$, points $x_1<\cdots <x_n$ in $[a,b]$ with $-\infty<a<b<\infty$, and points $y_1<\cdots<y_n$ (all different from $x_1,\ldots,x_n$) in $[c,d]$ with $-\infty<c<d<\infty$ such that  
\begin{equation} 
(L_n)_{jk} = \frac{r_j-s_k}{x_j-y_k}, \qquad 1\leq j,k\leq N. 
\label{eq:LoewnerMatrix} 
\end{equation}
In the special case when $y_j=-x_j$ and $s_j = -r_j$, a L\"{o}wner matrix is a Pick matrix (see Section~\ref{sec:PickMatrices}).
L\"{o}wner matrices satisfy the Sylvester matrix equation given by 
\[
D_{\underline{x}}L_n-  L_n D_{\underline{y}} = \underline{r}\,\underline{e}^T - \underline{e}\,\underline{s}^T,
\]
where $\underline{e} = (1,\ldots,1)^T$. From Theorem~\ref{thm:ratio} we can bound the singular values 
of $L_n$ provided that $[a,b]$ and $[c,d]$ are disjoint, i.e., either $b<c$ or $d<a$.  We emphasis that the
separation condition of the intervals $[a,b]$ and $[c,d]$ if an extra assumption on a L\"{o}wner matrix that allows 
us to proceed with the methodology we have developed.  
\begin{corollary} 
Let $L_n$ be an $n\times n$ L\"{o}wner matrix in~\eqref{eq:LoewnerMatrix} with $b<c$ or $d<a$. Then, for $j\geq 1$ we have
\[
\sigma_{j+2k}(L_n) \leq 4\left[\exp\left(\frac{\pi^2}{4\mu(1/\sqrt{\gamma})}\right)\right]^{-2k}\sigma_j(L_n), \qquad 1\leq j+2k\leq n, 
\]
where $\gamma$ is the absolute value of the cross-ratio of $a$, $b$, $c$, and $d$ (see~\eqref{eq:gamma}). If $a = c$ and $b = d$, then $2\mu(1/\sqrt{\gamma}) = \mu(a/b)$. The bound remains valid, but is slightly 
weaken, if $4\mu(1/\sqrt{\gamma})$ is replaced by $2\log(16\gamma)$. 
\label{cor:LoewnerMatrices}
\end{corollary} 
\begin{proof}
The same argument as in Corollary~\ref{cor:CauchyMatrices}, but with $\displacementRank=2$. 
\end{proof} 
Corollary~\ref{cor:LoewnerMatrices} shows that many L\"{o}wner matrices can be well-approximated 
by low rank matrices with ${\rm rank}_\epsilon(L_n) = \mathcal{O}(\log \gamma\log(1/\epsilon))$ and are exponentially ill-conditioned. 

\section{The singular values of Krylov, Vandermonde, and Hankel matrices}\label{sec:HankelMatrices}
The three types of matrices considered in Section~\ref{sec:PickCauchyLoewnerMatrices} allowed for direct
applications of Theorem~\ref{thm:ratio} and Corollary~\ref{cor:Zbound}. In this section, we consider the 
more challenging tasks of bounding the singular values of Krylov matrices with Hermitian arguments, real Vandermonde matrices, 
and real positive definite Hankel matrices. 

\subsection{Krylov and real Vandermonde matrices}\label{sec:VandermondeMatrices}
An $m \times n$ matrix $K_{m,n}$ with $m\geq n$ is said to be a Krylov matrix with Hermitian argument if there exists a Hermitian
matrix $A\in\mathbb{C}^{m\times m}$ and a vector $\underline{w}\in\mathbb{C}^{m\times 1}$ such that
\begin{equation}
K_{m,n} = \begin{bmatrix} \underline{w} \,\Bigg| \, A\underline{w}\, \Bigg| \, \cdots \, \Bigg| \,  A^{n-1}\underline{w}\end{bmatrix}.
\label{eq:KrylovMatrices}
\end{equation} 
Vandermonde matrices of size $m\times n$ with real abscissas $\underline{x}\in\mathbb{R}^{m\times 1}$, i.e., $(V_{m,n})_{jk} = x_j^{k-1}$, are also Krylov 
matrices with $A=D_{\underline{x}}$ and $\underline w=(1,\ldots,1)^T$. Krylov matrices satisfy the following Sylvester matrix equation:
\begin{equation}
A K_{m,n} - K_{m,n} Q = \underline{s}\,\underline{e}_n^T, \qquad Q = \begin{bmatrix}0 & & & -1 \\ 1 \\ & \ddots \\&& 1 & 0 \end{bmatrix},
\label{eq:KrylovDisplacement} 
\end{equation} 
where $\underline{s}\in\mathbb{C}^{m\times 1}$ and $\underline{e}_n = \left(0,\ldots,0,1\right)^T$. Since 
$A$ is a normal matrix, we attempt to use Theorem~\ref{thm:ratio} to bound the singular values of $K_{m,n}$. 

For the analysis that follows, we require that $n$ is an even integer. This is not a loss of generality because 
by the interlacing theorem for singular values~\cite{Thompson_72_01}. To see this, let $K_{m,n-1}$ be the $m\times (n-1)$ 
Krylov matrix obtained from $K_{m,n}$ by removing its last column.  If $n$ is odd, then\footnote{Observe that the singular values of a matrix decrease when removing a column and thus $\sigma_{j}(K_{m,n-1})\leq \sigma_{j}(K_{m,n})$. Let $Y$ be a best rank $j+k-2$ approximation to $K_{m,n-1}$ so that $\sigma_{j+k-1}(K_{m,n-1}) = \| K_{m,n-1} - Y \|_2$ and consider $X$ obtained from $Y$ by concatenating (on the right) the last column of $K_{m,n}$. Then, the rank of $X$ is at most $j+k-1$ and hence, $\sigma_{j+k}(K_{m,n}) \leq \| K_{m,n} - X \|_2 = \| K_{m,n-1} - Y \|_2=\sigma_{j+k-1}(K_{m,n-1})$.}
\begin{equation}
\frac{\sigma_{j+k}(K_{m,n})}{\sigma_j(K_{m,n})} \leq \frac{\sigma_{j+k-1}(K_{m,n-1})}{\sigma_j(K_{m,n-1})}, \qquad 2\leq j+k\leq n,
\label{eq:SingularValuesInterlacing}
\end{equation} 
and one can bound $\sigma_{j+k-1}(K_{m,n-1})/\sigma_j(K_{m,n-1})$ instead. From now on in this section 
we will assume that $n$ is an even integer. 

The Sylvester matrix equation in~\eqref{eq:KrylovDisplacement} contains matrices $A$ and $Q$, which are both normal matrices.
The eigenvalues of $A$ are contained in $\mathbb{R}$ and the eigenvalues of $Q$ are the $n$ (shifted) roots of unity, i.e., 
\[
\sigma(Q) = \left\{z\in\mathbb{C} : z = e^{ \frac{2\pi i (j+1/2)}{n} }, 0\leq j\leq n-1\right\}.
\]
Since $n$ is even, the spectrum of $Q$ and the real line are disjoint. Using Theorem~\ref{thm:ratio} we find that for $j\geq 1$ and $1\leq j+k\leq n$
\[
\sigma_{j+k}(K_{m,n}) \leq Z_k(E,F) \sigma_{j}(K_{m,n}), \qquad E \subseteq \mathbb{R}, \quad F = F_+\cup F_{-},
\]
where $F_+$ and $F_-$ are complex sets defined by 
\begin{equation}
F_+ = \{ e^{it}  : t \in [\tfrac{\pi}{n}, \pi - \tfrac{\pi}{n}]\}, \quad F_- = \{ e^{it} : t \in [- \pi + \tfrac{\pi}{n}, - \tfrac{\pi}{n}]\}.
\label{eq:FplusFminus}
\end{equation} 
Figure~\ref{fig:ZolotarevForHankel} shows the two sets $E$ and $F$ in the complex plane. As $n\rightarrow \infty$ the sets $F_+$ 
and $F_-$ approach the real line, suggesting that our bound on the singular values must depend on $n$ somehow.  
Our task is to bound the quantity $Z_k(E,F_+\cup F_-)$ --- a Zolotarev number that is not immediately related to one of the form 
$Z_k([-b,-a],[a,b])$. 

\begin{figure} 
\centering
\begin{tikzpicture} 
\tikzset{cross/.style={cross out, draw=black, thick, minimum size=2*(#1-\pgflinewidth), inner sep=0pt, outer sep=0pt},
cross/.default={3pt}};
\draw[ultra thick, black]( 5.9754, 3.3129) arc (9.000:171.000:2);
\draw[ultra thick, black]( 2.0246, 2.6871) arc (189.000:351.000:2);
\draw[fill,black] (5.9754,2.6871) circle (.1cm);
\draw[fill,black] (5.7820,2.0920) circle (.1cm);
\draw[fill,black] (5.4142,1.5858) circle (.1cm);
\draw[fill,black] (4.9080,1.2180) circle (.1cm);
\draw[fill,black] (4.3129,1.0246) circle (.1cm);
\draw[fill,black] (3.6871,1.0246) circle (.1cm);
\draw[fill,black] (3.0920,1.2180) circle (.1cm);
\draw[fill,black] (2.5858,1.5858) circle (.1cm);
\draw[fill,black] (2.2180,2.0920) circle (.1cm);
\draw[fill,black] (2.0246,2.6871) circle (.1cm);
\draw[fill,black] (2.0246,3.3129) circle (.1cm);
\draw[fill,black] (2.2180,3.9080) circle (.1cm);
\draw[fill,black] (2.5858,4.4142) circle (.1cm);
\draw[fill,black] (3.0920,4.7820) circle (.1cm);
\draw[fill,black] (3.6871,4.9754) circle (.1cm);
\draw[fill,black] (4.3129,4.9754) circle (.1cm);
\draw[fill,black] (4.9080,4.7820) circle (.1cm);
\draw[fill,black] (5.4142,4.4142) circle (.1cm);
\draw[fill,black] (5.7820,3.9080) circle (.1cm);
\draw[fill,black] (5.9754,3.3129) circle (.1cm);
\draw[ultra thick, black] (1,3) -- (7,3);
\draw[ultra thick, black, dashed] (7,3) -- (8,3) node[anchor=south] {$\sigma(A)\subset E \subseteq\mathbb{R}$};
\draw[ultra thick, black, dashed] (0,3) -- (1,3);
\node[] at (6,4.5) (a) {$F_{+}$};
\node[] at (6,1.5) (a) {$F_{-}$};
\node[] at (8.5,2) (a) {$\sigma(Q)\subset F = F_{+}\cup F_{-}$};
\node[] at (4.2,2.8) (a) {$\mathbf{0}$};
\draw (4,3) node[cross,rotate=0] {};
\end{tikzpicture}
\put(-288,113){\footnotesize ${\rm Im}$}
\put(-270,98){\footnotesize ${\rm Re}$}
\put(-278,105){$\longrightarrow$}
\put(-275,101){\rotatebox{90}{ $\longrightarrow$}}
\caption{The sets $E$ and $F$ in the complex plane for the Zolotarev problem~\eqref{eq:zolotarev} used to bound 
the singular values of a $20\times 20$ Krylov matrix with a Hermitian argument. The sets $F_{+}$ and $F_{-}$ are a distance of only $\mathcal{O}(1/n)$ from 
the real axis, where $n$ is the size of the Krylov matrix, and this causes the $\log n$ dependence in the weaken version of~\eqref{eq:KrylovBound3}. 
The solid black dots denote the spectrum of $Q$, which is contained in $F_{+}\cup F_{-}$.}
\label{fig:ZolotarevForHankel}
\end{figure}
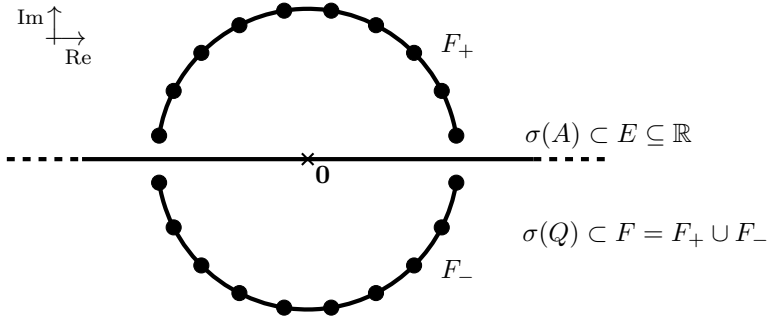 

The following lemma relates the quantity $Z_{2k}(E, F_+\cup F_-)$ to the Zolotarev number $Z_k([-1/\ell,-\ell], [\ell,1/\ell])$ with $\ell = \tan(\pi/(2n))$: 
\begin{lemma}
Let $k\geq 1$ be an integer and $E\subseteq \mathbb{R}$.  Then, $Z_{2k+1}(E,F_+\cup F_-)\leq Z_{2k}(E,F_+\cup F_-)$ and 
\[
Z_{2k}(E,F_+\cup F_-) \leq \frac{2\sqrt{Z_{k}}}{1+Z_{k}}, \quad Z_k := Z_k([-1/\ell,-\ell], [\ell,1/\ell]),
\] 
where $\ell = \tan(\pi/(2n))$, the complex sets $F_+$ and $F_-$ are defined in~\eqref{eq:FplusFminus}, and $n$ is an even integer. 
\label{lem:ZolotarevRelation} 
\end{lemma} 
\begin{proof} 
Let $R(z)\in \mathcal{R}_{k,k}$ be the extremal function for $Z_{k}:=Z_{k}([-1/\ell,-\ell],[\ell,1/\ell])$ characterized in Theorem~\ref{thm:ZextremalRational}, where $\ell = \tan(\pi/(2n))$. Since the M\"{o}bius transform given by 
\[
T(z) = \frac{1}{i} \frac{z-1}{z+1}
\]
maps $F_+$ to $[\ell,1/\ell]$, $F_-$ to $[-1/\ell,-\ell]$, and $\mathbb{R}$ to $i\mathbb{R}$, we have 
\[
Z_{2k}(\mathbb{R},F_+\cup F_-) = Z_{2k}(i\mathbb{R}, [-1/\ell,-\ell] \cup [\ell,1/\ell]) = \inf_{r\in\mathcal{R}_{2k,2k}}\frac{\sup_{z\in\mathbb{R}} |r(iz)|}{\inf_{z\in[-1/\ell,-\ell] \cup [\ell,1/\ell]} |r(z)|}.
\]
Now, consider the rational function 
\[
S(z) = \frac{R(z) + 1/R(z)}{2} = \frac{R(z)+R(-z)}{2} \in \mathcal{R}_{2k,2k}, 
\]
where we used the fact that $1/R(z) = R(-z)$ (see Theorem~\ref{thm:ZextremalRational}, (b)). 
Since $|R(iz)|=1$ for $z\in\mathbb{R}$ (seeTheorem~\ref{thm:ZextremalRational}, (c)), we have
\[
\sup_{z\in\mathbb{R}} \left|S(iz)\right|  = \sup_{z\in\mathbb{R}} \left|\frac{R(iz) + R(-iz)}{2}\right| \leq 1.
\]
Moreover, since $-1\leq-\sqrt{Z_k}\leq R(z)\leq \sqrt{Z_k}\leq 1$ for $z\in [-1/\ell,-\ell]$ (see Theorem~\ref{thm:ZextremalRational}, (a)) and $x\mapsto 2x/(1+x^2)$ is a nondecreasing function 
on $[-1,1]$ and $S(-z)=S(z)$, we have 
\[
\begin{aligned}
\inf_{z\in[-1/\ell,-\ell] \cup [\ell,1/\ell]} \left|S(z)\right| &= \sup_{z\in[-1/\ell,-\ell]} \left|\frac{2}{R(z)+1/R(z)}\right| \\
&= \sup_{z\in[-1/\ell,-\ell]} \left|\frac{2R(z)}{1+R(z)^2}\right|\\
&\leq \frac{2\sqrt{Z_k}}{1+Z_k}.
\end{aligned}
\]
Therefore, $Z_{2k}(E,F_+\cup F_-) \leq Z_{2k}(\mathbb{R},F_+\cup F_-) \leq 2\sqrt{Z_k}/(1+Z_k)$ as required. The bound 
$Z_{2k+1}(E,F_+\cup F_-)\leq Z_{2k}(E,F_+\cup F_-)$ trivially holds from the definition of Zolotarev numbers. 
\end{proof}

By Corollary~\ref{cor:Zbound} we have the slightly weaker upper bound for $Z_{2k}(E,F_+\cup F_-)$:
\[
Z_{2k}(E,F_+\cup F_-) \leq 2\sqrt{Z_k([-1/\ell,-\ell],[\ell,1/\ell])} \leq 4 \rho^{-k}, 
\]
where since $\tan x \geq x$ for $0\leq x\leq \pi/2$, we have 
\[
\rho = {\rm exp}\!\left(\frac{\pi^2}{2\mu(\tan(\pi/(2n))^2)}\right) \geq  \exp\!\left(\frac{\pi^2}{2\log( 4 / \tan(\pi/(2n))^2 )}\right) \geq \exp\!\left(\frac{\pi^2}{4\log(4n/\pi)}\right)\!.
\]
If $n$ is an even integer, then we can immediately conclude a bound on the singular values from Theorem~\ref{thm:ratio}. If $n$ is an odd integer, 
then one must employ~\eqref{eq:SingularValuesInterlacing} first.
\begin{corollary} 
The singular values of $K_{m,n}$ can be bounded as follows:
\begin{equation}
\sigma_{j+2k}(K_{m,n}) \leq 4\!\left[\exp\!\left(\frac{\pi^2}{2\mu(\tan(\pi/(4\lfloor n/2\rfloor))^2)}\right)\right]^{-k+[n]_2} \sigma_{j}(K_{m,n}), \quad 1\leq j+2k\leq n,
\label{eq:KrylovBound3} 
\end{equation} 
where $\mu(\cdot)$ is the Gr\"{o}tzsch function and $[n]_2=1$ if $n$ is odd and is $0$ if $n$ is even. 
The bound above remains valid, but is slightly weaken, if $2\mu(\tan(\pi/(4\lfloor n/2\rfloor))^2)$ is replaced by $4\log(8\lfloor n/2\rfloor/\pi)$.
\label{cor:KrylovBounds} 
\end{corollary} 

Figure~\ref{fig:KrylovMatrices} demonstrates the bound on the singular values in~\eqref{eq:KrylovBound3} on $n\times n$ 
Krylov matrices, where $n = 10, 100, 1000$. The step behavior of the bound is due to the fact that~\eqref{eq:KrylovBound3} 
only bounds $\sigma_{1+2k}(K_{m,n})$ when $n$ is even and we use the trivial inequality $\sigma_{2k+2}(K_{m,n})\leq \sigma_{2k+1}(K_{m,n})$ otherwise. 
One also observes that the singular values of Krylov matrices with Hermitian arguments can
decay at a supergeometric rate; however, the analysis in this paper only realizes a geometric decay. 
Therefore,~\eqref{eq:KrylovBound3} is only a reasonable bound on $\sigma_j(K_{m,n})$ when $j$ is a small integer with respect 
to $n$. If $j/n\rightarrow c$ and $c\in(0,1)$, then improved bounds on $\sigma_j(K_{m,n})$ may be possible by bounding discrete 
Zolotarev numbers~\cite{BB_Gryson}.   The bound in~\eqref{eq:KrylovBound3} provides an upper bound on 
the $\epsilon$-rank of $K_{m,n}$:
\[
{\rm rank}_\epsilon(K_{m,n}) \leq  2\bigg\lceil \frac{4\log\left(8\lfloor n/2\rfloor/\pi\right)\log\left(4/\epsilon\right)}{\pi^2} \bigg\rceil + 2,
\]
which allows for either an odd or even integer $n$. 
\begin{figure} 
\begin{minipage}{.49\textwidth} 
\begin{overpic}[width=\textwidth]{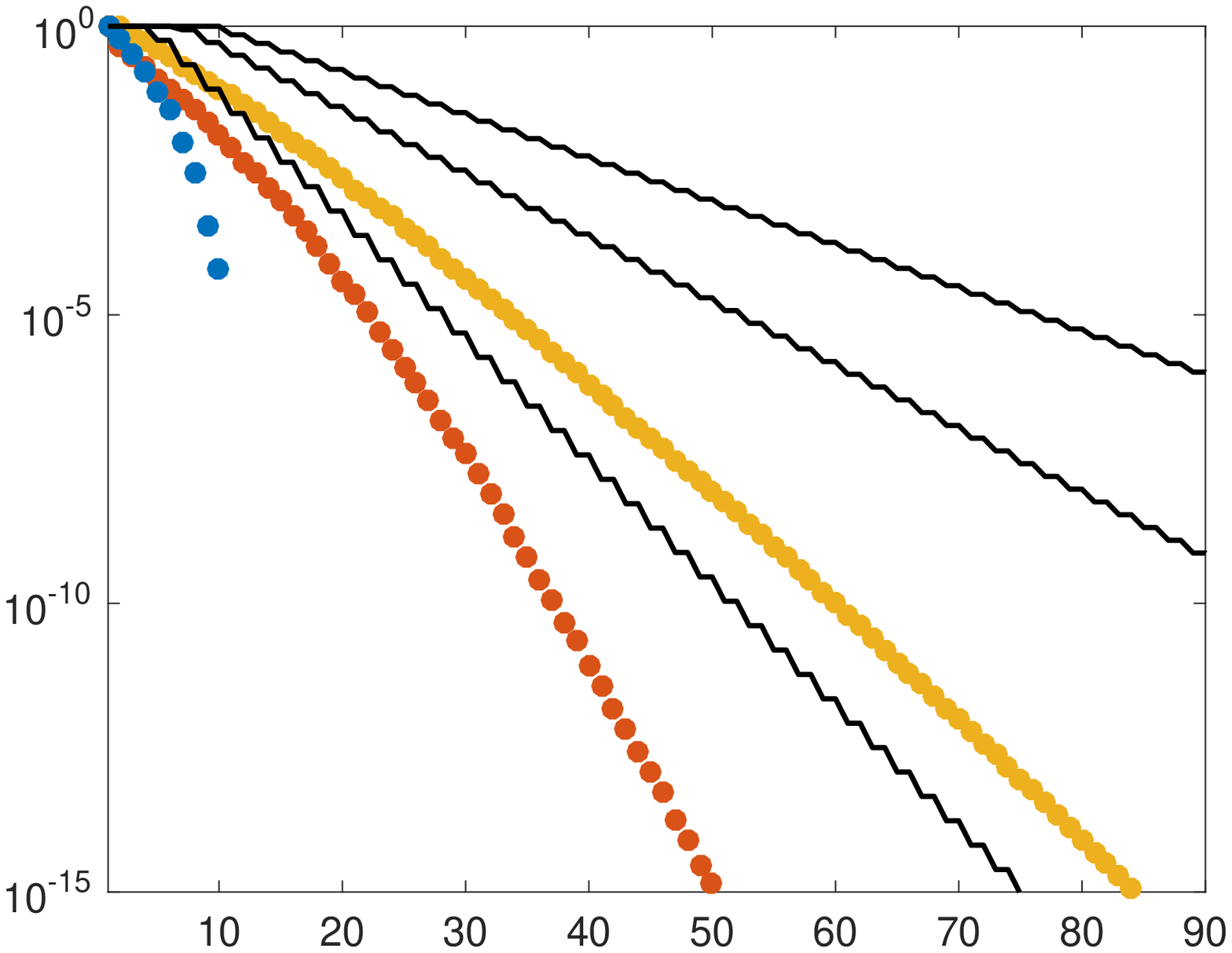}  
\put(47,0) {Index}
\put(54,25) {\rotatebox{-45}{$n = 10$}}
\put(69,38) {\rotatebox{-32}{$n = 100$}}
\put(69,48) {\rotatebox{-22}{$n = 1000$}}
\end{overpic}  
\end{minipage} 
\begin{minipage}{.49\textwidth}
\begin{overpic}[width=\textwidth]{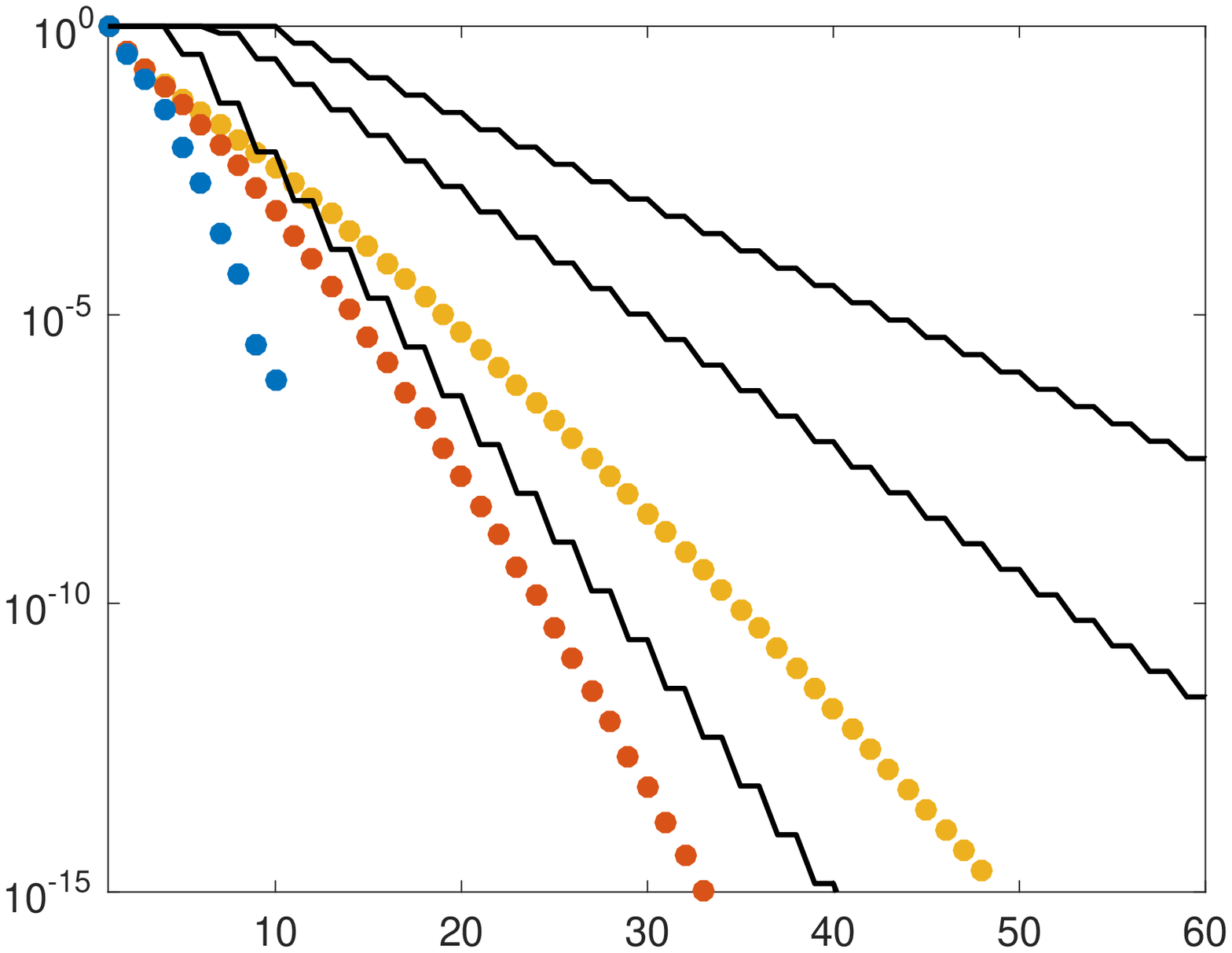}  
\put(47,0) {Index}
\put(54,25) {\rotatebox{-56}{$n = 10$}}
\put(69,38) {\rotatebox{-36}{$n = 100$}}
\put(69,50) {\rotatebox{-25}{$n = 1000$}}
\end{overpic}  
\end{minipage} 
\caption{Left: The singular values of $n\times n$ Krylov matrix (colored dots) compared to the bound in~\eqref{eq:KrylovBound3} for $n = 10$ (blue),$100$ (red), $1000$ (yellow). In~\eqref{eq:KrylovMatrices} the matrix $A$ is a diagonal matrix with
entries taken to be equally spaced points in $[-1,1]$ and $\underline{w}$ is a random vector with independent 
Gaussian entries. Right: The singular values of the $n\times n$ real positive definite Hankel matrices (colored dots) associated to the measure $\mu_H(x) = \mathbf{1}|_{-1\leq x\leq 1}$ 
compared to the bound in~\eqref{eq:HankelBound} for $n = 10$ (blue),$100$ (red), $1000$ (yellow).}
\label{fig:KrylovMatrices} 
\end{figure} 

Recall that Vandermonde matrices with real abscissas are also Krylov matrices with Hermitian arguments.
Therefore, the bounds in this section also apply to Vandermonde matrices with real abscissas and shows that 
they have rapidly decaying singular values and are exponentially ill-conditioned. An observation that has been 
extensively investigated in the literature~\cite{b46,gaut2,Pan_16_01}.

\subsection{Real positive definite Hankel matrices} \label{sec:PosDefHankel}
An $n\times n$ matrix $H_n$ is a Hankel matrix if the matrix is constant along each anti-diagonal, i.e., $(H_n)_{jk} = h_{j+k}$ for $1\leq j,k\leq n$. 
Clearly, not all Hankel matrices have decaying singular values, for example, the exchange matrix 
has repeated singular values of $1$.  This means that any displacement structure that is satisfied by all Hankel matrices, for example, 
\[
{\rm rank}\left( QX - XQ^T \right) \leq 2,
\]
where $Q$ is given in~\eqref{eq:KrylovDisplacement}, does not result in a Zolotarev number that decays. 
Motivated by the Hilbert matrix in Section~\ref{sec:CauchyMatrices}, we show that every real and positive 
definite Hankel matrix has rapidly decaying singular values.  Previous work has led to bounds that can be 
calculated by using a pivoted Cholesky algorithm~\cite{Antoulas_02_01}, bounds for very special 
cases~\cite{Townsend_16_01}, as well as incomplete attempts~\cite{tyrt98,tyrt}. 

In order to exploit the positive definite structure we recall that the Hamburger moment problem states 
that a real Hankel matrix is positive semidefinite if and only if it is associated to a nonnegative Borel measure 
supported on the real line. 
\begin{lemma}
 A real $n \times n$ Hankel matrix, $H_n$, is positive semidefinite if and only if there exists a 
 nonnegative Borel measure $\mu_H$ supported on the real line such that
 \begin{equation}
  (H_n)_{jk} = \int_{-\infty}^{\infty} x^{j+k-2} \mathrm{d}\mu_H(x), \qquad 1\leq j,k\leq n.
  \label{eq:moments}
 \end{equation}
 \end{lemma}
 \begin{proof}
  For a proof, see~\cite[Theorem~7.1]{Peller_12_01}. 
 \end{proof}

Let $H_n$ be a real positive definite Hankel matrix associated to the nonnegative weight $\mu_H$ in~\eqref{eq:moments} supported on $\mathbb{R}$. 
Let $x_1,\ldots, x_n$ and $w_1^2,\ldots,w_n^2$ be the Gauss quadrature nodes and weights associated to $\mu_H$. Then, since a Gauss quadrature is exact for polynomials of degree $2n-1$ or less, we have 
\[
(H_n)_{jk} = \int_{-\infty}^{\infty} x^{j+k-2} \mathrm{d}\mu_H(x) = \sum_{s = 1}^n w_s^2 x_s^{j+k-2} = \sum_{s = 1}^n (w_s x_s^{j-1})(w_s x_s^{k-1}).
\]
Therefore, every real positive definite Hankel matrices has a so-called {\em Fiedler factorization}~\cite{fiedler}, i.e., 
\[
H_n = K_{n,n}^*K_{n,n}, \qquad K_{n,n} = \begin{bmatrix} \underline{w} \,\Bigg| \, D_{\underline{x}}\underline{w}\, \Bigg| \, \cdots \, \Bigg| \,  D_{\underline{x}}^{n-1}\underline{w}\end{bmatrix}\!,
\]
where $K_{n,n}$ is a Krylov matrix with Hermitian argument and $K_{n,n}^*$ is the conjugate transpose of $K_{n,n}$. 
This means that $\sigma_j(H_n) = \sigma_j(K_{n,n})^2$ for $1\leq j\leq n$.  That is, a 
bound on the singular values of $H_n$ and the $\epsilon$-rank of $H_n$ directly follows from~\eqref{eq:KrylovBound3}. 
\begin{corollary}
Let $H_n$ be an $n\times n$ real positive definite Hankel matrix. Then, 
\begin{equation}
\sigma_{j+2k}(H_n) \leq 16 \left[\exp\left(\frac{\pi^2}{4\log(8\lfloor n/2\rfloor/\pi)}\right)\right]^{-2k+2} \sigma_{j}(H_n), \qquad 1\leq j+2k\leq n,
\label{eq:HankelBound}
\end{equation} 
and 
\[
{\rm rank}_\epsilon(H_n) \leq 2\bigg\lceil \frac{2\log\left(8\lfloor n/2\rfloor/\pi\right)\log\left(16/\epsilon\right)}{\pi^2}\bigg\rceil+2,
\]
where both bounds allow for $n$ to be an even or odd integer. 
\end{corollary} 

We conclude that all real positive definite Hankel matrices have an $\epsilon$-rank of at most $\mathcal{O}(\log n\log(1/\epsilon))$, explaining why
low rank techniques are usually advantageous in computational mathematics on such matrices. 

Since a real positive semidefinite Hankel matrix can be arbitrarily approximated by a real positive definite Hankel matrix, the 
results from this section immediately extend to such Hankel matrices.\footnote{For a real positive semidefinite Hankel matrix one may improve our bounds on the singular values of $H_n$ by replacing $n$ by the rank of $H_n$.} This fact was exploited, but not proved in general, in~\cite{Townsend_16_01} 
to derive quasi-optimal complexity fast transforms between orthogonal polynomial bases. 

\section*{Acknowledgments}
Many of the inequalities in this paper were numerically verified using RKToolbox~\cite{rktoolbox} and we thank Stefan G\"{u}ttel
for providing support. We thank Yuji Nakatsukasa for carefully reading a draft of this manuscript and pointing us 
towards~\cite{Guttel_14_01} and~\cite{Nakatsukasa_16_01}. We also thank Sheehan Olver, Gil Strang, Marcus Webb, and Heather Wilber for discussions.

\appendix \section{Typos in an infinite product formula}\label{appendix:correction}
In Section~\ref{sec:Zolotarev} we noted that there were typos in an infinite product formula given by Lebedev~\cite[(1.11)]{Lebedev_77_01}. 
The mistake has unfortunately been copied several times in the literature. Here, we attempt to correct these typos. 

Lebedev~\cite{Lebedev_77_01} and his successors~\cite{Medovikov_05_01,Oseledets_07_01} were not concerned with the Zolotarev problem in~\eqref{eq:zolotarev}, 
but instead the equivalent problem of minimal Blaschke products in the half plane, i.e.,  
\begin{equation}
E_k([a,b]) = \min_{z_1,...,z_k\in \mathbb{C}} \max_{z\in [a,b]} \left| \prod_{s=1}^k \frac{z-z_s}{z+\overline{z_s}} \right|, \qquad 0<a<b<\infty.
\label{eq:Ek}
\end{equation} 
In~\cite[(1.11)]{Lebedev_77_01}, Lebedev presented an infinite product formula for $E_k$ that unfortunately contained typos and 
resulted in an erroneous lower bound for $E_k$ in~\cite[(1.12)]{Lebedev_77_01}. More recently, other erroneous lower bounds have been claimed 
in~\cite[(4.1)]{Guttel_14_01} for a related problem based on~\cite[(3.17)]{Medovikov_05_01}.  

To correct the situation we first show that with $Z_k := Z_k([-b,-a],[a,b])$ we have 
\[
 \sqrt{Z_k} = E_k([-b,-a])=E_k([a,b])=E_k([a/b,1]), 
\]
where the last two equalities are immediate from symmetry considerations and scaling. Since any $z_1,\ldots,z_k\in\mathbb{C}$ 
describes a rational function for $E_k([-b,-a])$ in~\eqref{eq:Ek}, the solution to~\eqref{eq:Ek} describes a rational function that is a candidate for the Zolotarev problem in~\eqref{eq:zolotarev} and we 
have $\sqrt{Z_k} \leq E_k([-b,-a])$. Conversely, taking $R(z)$ as in Theorem~\ref{thm:ZextremalRational} we get from property (c) that $R(z)$ has a set of poles being closed under complex conjugation. Property (b) tells us that, if $p_j$ is a pole of $R$, then $-p_j$ is a zero of $R$. Thus, from Theorem~\ref{thm:ZextremalRational} we have
\[
   R(z) = \pm \prod_{j=1}^k \frac{z+\overline{p_j}}{z-p_j},
\]
which implies that $E_k([-b,-a])\leq \max_{z\in [-b,-a]} | R(z)| \leq \sqrt{Z_k}$.  Here, in the last inequality we have applied property (a). We conclude that 
$E_k([-b,-a]) = \sqrt{Z_k}$.

Therefore, an infinite product formula for $E_k([\eta,1])$ that corrects~\cite[(1.11)]{Lebedev_77_01} is obtained by taking square roots (and setting $a/b=\eta$) in Theorem~\ref{thm:Zproduct}.  That is, for $0<\eta<1$ we have
\begin{equation} \label{eq:Eproduct}
  E_k([\eta,1]) = 2\rho^{-k}\prod_{\tau=1}^\infty \frac{(1+\rho^{-8\tau k})^2}{(1+\rho^{4k}\rho^{-8\tau k})^2}, \qquad \rho=\exp\left(\frac{\pi^2}{2\mu(\eta)}\right),
\end{equation} 
where $\mu(\cdot)$ is the Gr\"{o}tzsch ring function.  From~\eqref{eq:Eproduct}, one obtains upper and lower bounds for $E_k([\eta,1])$ that correct~\cite[(1.12)]{Lebedev_77_01},~\cite[(3.17)]{Medovikov_05_01}, and~\cite[(15)]{Oseledets_07_01}, namely
\begin{equation}
\frac{2\rho^{-k}} {(1+\rho^{-4k})^2} \leq E_k([\eta,1]) \leq \frac{2\rho^{-k}}{1+\rho^{-4k}} \leq 2\rho^{-k}, \qquad \rho=\exp\left(\frac{\pi^2}{2\mu(\eta)}\right).
\label{eq:Ebounds}
\end{equation}
More refined estimates than in~\eqref{eq:Ebounds} can be obtained by taking more terms from the infinite product in~\eqref{eq:Eproduct}. 

More recently, the best rational approximation of the sign function on $[-b,-a]\cup [a,b]$ has become important in numerical linear algebra because of a recursive 
construction of spectral projectors of matrices~\cite{Guttel_14_01,Nakatsukasa_16_01}.  In this setting, if $E_{m,n} :=E_{m,n}([-b,-a],[a,b])$ then
\[
E_{m,n} = \min_{r\in \mathcal R_{m,n}} \max_{z\in [-b,-a]\cup[a,b]} \left| r(z) - {\rm sgn}(z) \right|,\quad {\rm sgn}(z) = \begin{cases}1,&z\in [a,b],\\ -1,&z\in [-b,-a].\end{cases}
\]
Unfortunately, lower and upper bounds for $E_{2k,2k} = E_{2k-1,2k}$ are claimed in~\cite[(4.1)]{Guttel_14_01} based 
on the erroneous infinite product formula in~\cite[(3.17)]{Medovikov_05_01} and for $E_{2k+1,2k+1}=E_{2k+1,2k}$ in~\cite[(3.8)]{Nakatsukasa_16_01} 
by incorrectly citing the fundamental work of Gon\v{c}ar~\cite[(32)]{Goncar_69_01}.

We believe it is therefore useful to state infinite product formulas for $E_{k,k}$ and the resulting estimates. We recall from the proof of Theorem~\ref{thm:Zproduct} and Theorem~\ref{thm:ZextremalRational} that we have 
\[
E_{k,k}  = E_{2 \lfloor (k-1)/2 \rfloor +1, 2 \lfloor k/2 \rfloor}  =  \frac{2\sqrt{Z_k}}{1+Z_k}, \qquad   \mu\left(\frac{2\sqrt{Z_k}}{1+Z_k}\right) = \frac{\mu(Z_k)}{2}. 
\]
Thus, in the proof of Theorem~\ref{thm:Zproduct} we select $q=\exp(- 2\mu(E_{k,k})) = \rho^{-2k}$ and obtain
\begin{equation}
E_{k,k} =   4 \rho^{-k} \prod_{\tau=1}^\infty  \frac{ ( 1+\rho^{-4\tau k})^4 }{( 1+\rho^{2k} \rho^{-4\tau k})^4 },\qquad \rho=\exp\left(\frac{\pi^2}{2\mu(a/b)}\right).
\label{eq:sgnProduct}
\end{equation} 
Again, this infinite product in~\eqref{eq:sgnProduct} results in asymptotically tight corrected lower and upper bounds on $E_{k,k}$: 
\begin{equation} 
 \frac{4 \rho^{-k}}{(1+\rho^{-2k})^4} \leq E_{k,k} \leq \frac{4 \rho^{-k}}{(1+\rho^{-2k})^2} \leq 4 \rho^{-k}, \qquad \rho=\exp\left(\frac{\pi^2}{2\mu(a/b)}\right).
\label{eq:sgnBounds} 
\end{equation}
Similarly, more refined estimates than in~\eqref{eq:sgnBounds} can be obtained by taking more terms from the infinite product in~\eqref{eq:sgnProduct}. 

\end{document}